\documentclass[12pt]{amsart}

\usepackage{vmargin}            
\usepackage[applemac]{inputenc}
\usepackage[english]{babel}
\usepackage{enumerate}
\usepackage{amsmath, csquotes, bm, bbm, amssymb, mathrsfs, mathabx, dsfont}
\usepackage{amsfonts,amssymb,amsthm}
\usepackage{mathtools}
\usepackage{commath}
\usepackage{hyperref}
\usepackage{braket}
\usepackage{xcolor}
\usepackage{mdwlist}

\newtheorem{theorem}{Theorem}[section]
\newtheorem{lemma}[theorem]{Lemma}
\newtheorem{proposition}[theorem]{Proposition}
\newtheorem{corollary}[theorem]{Corollary}
\theoremstyle{definition}
\newtheorem{notation}[theorem]{Notation}
\theoremstyle{definition}
\newtheorem{definition}[theorem]{Definition}
\theoremstyle{remark}
\newtheorem{remark}[theorem]{Remark}

\newtheorem*{remark*}{Remark}
\numberwithin{equation}{section}
\newenvironment{proof-sketch}{\textit{Sketch of proof.~}}{\quad}

\newcommand{\R}{\mathbb{R}}
\newcommand{\C}{\mathbb{C}}
\newcommand{\N}{\mathbb{N}}
\newcommand{\dx}{\mathrm{\,d}}
\newcommand{\n}{\mathbf{n}}
\newcommand{\x}{\mathbf{x}}

\newcommand{\pa}{\partial}
\newcommand{\eps}{\varepsilon}
\renewcommand{\Re}{\mathrm{Re}}

\newcommand{\D}{\mathsf{Dom}}

\newcommand{\Bk}{\color{black}}

\title[An infinite mass limit]{The MIT Bag Model as an infinite mass limit}

\author[N. Arrizabalaga]{Naiara Arrizabalaga}
\address[N. Arrizabalaga]{Departamento de Matem\'aticas, Universidad del Pa\'is Vasco/Euskal Herriko Unibertsitatea (UPV/EHU), 48080 Bilbao, Spain}
\email{naiara.arrizabalaga@ehu.eus}

\author[L. Le Treust]{Lo\"ic Le Treust}
\address[L. Le Treust]{Aix Marseille Univ, CNRS, Centrale Marseille, I2M, Marseille, France}
\email{loic.le-treust@univ-amu.fr}

\author[A. Mas]{Albert Mas}
\address[A. Mas]{Departament de Matem\`atiques i Inform\`atica,
Universitat de Barcelona. Gran Via de les Corts Catalanes 585,
08007 Barcelona, Spain}
\email{albert.mas@ub.edu}

\author[N. Raymond]{Nicolas Raymond}
\address[N. Raymond]{D\'{e}partement de Math\'{e}matiques (LAREMA), Universit\'{e} d'Angers, Facult\'{e} des Sciences, 2, boulevard Lavoisier, 49045 Angers cedex 01, France}
\email{nicolas.raymond@univ-rennes1.fr}

\keywords{Dirac operator, relativistic particle in a box, MIT bag model, spectral theory}                                  
\subjclass[2010]{35J60, 35Q75, 49J45, 49S05, 81Q10, 81V05, 35P15, 58C40}

\begin{document}


\begin{abstract}
The Dirac operator, acting in three dimensions, is considered. Assuming that a large mass $m>0$ lies outside a smooth and bounded open set $\Omega\subset\R^3$, it is proved that its spectrum is approximated by the one of the Dirac operator on 
$\Omega$ with the MIT bag boundary condition. The approximation, which is developed up to and error of order $o(1/\sqrt m)$, is carried out by introducing tubular coordinates in a neighborhood of $\partial\Omega$ and analyzing the corresponding one dimensional optimization problems in the normal direction. 
\end{abstract}
\maketitle
\tableofcontents

\section{Introduction}
\subsection{Context}
This paper is devoted to the spectral analysis of the Dirac operator with high scalar potential barrier in three dimensions. More precisely, we will assume that there is a large mass $m$ outside a smooth and bounded open set $\Omega$. From physical considerations, see \cite{bogoliubov1987,MIT061974}, it is expected that, when $m$ becomes large, the eigenfunctions of low energy do not visit $\R^3\setminus\Omega$ and tend to satisfy the so-called MIT bag condition on $\partial\Omega$. This boundary condition, that we will define in the next section, is usually chosen by the physicists \cite{johnson,MIT061974,PhysRevD.12.2060}, in order to get a vanishing normal flux at the bag surface. It was originally introduced by Bogolioubov in the late $60's$ \cite{bogoliubov1987} to describe the confinement of the quarks in the hadrons with the help of an infinite scalar potential barrier outside a fixed set $\Omega$. In the mid $70's$, this model has been revisited into a shape optimization problem named \emph{MIT bag model} \cite{johnson,MIT061974,PhysRevD.12.2060} in which the optimized energy takes the form
\[
	\Omega\mapsto \lambda_1(\Omega) + b|\Omega|,
\] 
where $\lambda_1(\Omega)$ is the first nonnegative eigenvalue of the Dirac operator with the boundary condition introduced by Bogolioubov, $|\Omega|$ is the volume of $\Omega\subset\R^3$ and $b>0$.
The interest of the bidimensional equivalent of this model has recently been renewed with the study of graphene where this condition is sometimes called \enquote{infinite mass condition}, see \cite{PhysRevB.77.085423,MR901725}. The aim of this paper is to provide a mathematical justification of this terminology, and extend to dimension three the work \cite{stockmeyer2016infinite}.

\subsection{The Dirac operator with large effective mass}
In the whole paper,  $\Omega$ denotes a fixed bounded domain of $\R^3$ with regular boundary. The Planck's constant and the velocity of light are assumed to be equal to $1$.

Let us recall the definition of the Dirac operator associated with the energy of a relativistic particle of mass $m_0$ and spin $\frac{1}{2}$, see \cite{Thaller1992}. The Dirac operator is a first order differential operator $(H,\D(H))$, acting on $L^2(\R^3;\mathbb{C}^4)$ in the sense of distributions, defined by 
\begin{equation}
H=\alpha\cdot D+m_0\beta \,,\qquad D=-i\nabla\,,
\end{equation}
where $\D(H) = H^1(\R^3; \C^4)$,  $\alpha=(\alpha_1,\alpha_2,\alpha_3)$ and  $\beta$ 
are the $4\times4$ Hermitian and unitary matrices given by
\[
	\beta=\left(\begin{array}{cc}1_2&0\\0&-1_2\end{array}\right),\;
	\alpha_k=\left(\begin{array}{cc}0&\sigma_k\\\sigma_k&0\end{array}\right) \mbox{ for }k=1,2,3\,.	
\]
%
Here, the Pauli matrices $\sigma_1,\sigma_2$ and $\sigma_3$ are defined by
\[
	\sigma_1 = \left(
		\begin{array}{cc}
			0&1\\1&0
		\end{array}
	\right),
	\quad
	\sigma_2 = \left(
		\begin{array}{cc}
			0&-i\\i&0
		\end{array}
	\right),
	\quad
	\sigma_3 = \left(
		\begin{array}{cc}
			1&0\\0&-1
		\end{array}
	\right)\,,
\]
and $\alpha \cdot X$ denotes $\sum_{j=1}^3\alpha_j X_j$ for any $X = (X_1,X_2,X_3)$.

In this paper, we consider particles with large effective mass $m\gg m_0$ outside $\Omega$. Their kinetic energy is associated with the self-adjoint operator $(H_m, \D(H_m))$ defined by
\[
H_{m}=\alpha\cdot D+(m_0 + m\chi_{\Omega'})\beta \,,
\]
where $\Omega'$ is the complementary set of $\overline{\Omega}$, $\chi_{\Omega'}$ is the characteristic function of $\Omega'$ and $\D(H_m) = H^1(\R^3;\C^4)$.
%
%
\begin{notation}\label{not.def1}
	In the following, $\Gamma := \pa \Omega$ and for all $\x\in\Gamma$, $\n(\x)$ is the outward-pointing unit normal vector to the boundary, $L(\x)=d\n_{\x}$ denotes the second fundamental form of the boundary and
	\begin{equation*}
		\kappa(\x)=\mathsf{Tr}\, L(\x)\, \mbox{ and }K(\x) =  \det \, L(\x)
	\end{equation*}
	are the mean curvature and the Gauss curvature of $\Gamma$ respectilvely.
\end{notation}
\Bk

\begin{definition}\label{def.MIT}
	The MIT bag Dirac operator $(H^\Omega,\D(H^\Omega))$ is defined on the domain
	\[
		\mathsf{Dom}(H^\Omega) = \{\psi\in H^1(\Omega;\mathbb{C}^4)~:~\mathcal{B}\psi = \psi~\text{on}~\Gamma\}\,,\qquad\mbox{ with }\quad\mathcal{B}=-i\beta(\alpha\cdot \n)\,,
	\]
	by $H^\Omega\psi=H\psi$ for all $\psi\in\D(H^\Omega)$. Observe that the trace is well-defined by a classical trace theorem.
\end{definition}
\begin{notation} 
We denote by $\braket{\cdot,\cdot}$ the $\C^4$ scalar product (antilinear w.r.t. the left argument) and by $\braket{\cdot,\cdot}_{U}$ the $L^2$ scalar product on the set $U\subset\R^3$.
\end{notation}

\begin{notation} 
We define, for every $\n\in \mathbb{S}^2$, the orthogonal projections 
\begin{equation}\label{eq:boundary_projections}
	\Xi^\pm = \frac{1_4\pm\mathcal{B}}{2}
\end{equation}
 associated with the eigenvalues $\pm1$ of the matrix $\mathcal{B}$.
\end{notation}

\subsection{Squared operators, heuristics, and results}
The aim of this paper is to relate the spectra of $H_{m}$ and $H^\Omega$ in the limit $m\to+\infty$. 

\begin{notation}\label{not.lambda}
Let $(\lambda_k)_{k\in \N^*}$ and $(\lambda_{k,m})_{k\in \N^*}$
denote the increasing sequence of eigenvalues of the operator $|H^\Omega|$ and $|H_m|$, respectively, each one being repeated according to its multiplicity.  By the $\min-\max$ characterization, we have
\begin{equation*}
\begin{split}
	\lambda_k &= 	\inf_{\tiny
					\begin{array}{c}
						V\subset \D(H^\Omega),
						\\
						\dim V = k,
					\end{array}}
				\sup_{\tiny
					\begin{array}{c}
						 \varphi\in V,
						\\
						\norm{ \varphi}_{L^2(\Omega)} = 1,
				\end{array}}
				\norm{H^\Omega \varphi}_{L^2(\Omega)}
				\\
				&=
				\sup_{\tiny \{\psi_{1},\ldots, \psi_{k-1}\}\subset \D(H^\Omega),}
				\inf_{\tiny
				\begin{array}{c}
					 \varphi\in \mathrm{span}(\psi_{1},\ldots, \psi_{k-1})^\perp,
					\\
					\norm{ \varphi}_{L^2(\Omega)} = 1,
				\end{array}}
				\norm{H^\Omega  \varphi}_{L^2(\Omega)}  ,
\end{split}
\end{equation*}
and
\begin{equation*}
\begin{split}
	\lambda_{k,m} &= 	\inf_{\tiny
					\begin{array}{c}
						V\subset H^1(\R^3;\C^4),
						\\
						\dim V = k,
					\end{array}}
				\sup_{\tiny
					\begin{array}{c}
						 \varphi\in V,
						\\
						\norm{ \varphi}_{L^2(\R^3)} = 1,
				\end{array}}
				\norm{H_m  \varphi}_{L^2(\R^3)}
				\\
				&=
				\sup_{\tiny \{\psi_{1},\ldots, \psi_{k-1}\}\subset H^1(\R^3;\C^4),}
				\inf_{\tiny
				\begin{array}{c}
					 \varphi\in \mathrm{span}(\psi_{1},\ldots, \psi_{k-1})^\perp,
					\\
					\norm{ \varphi}_{L^2(\R^3)} = 1,
				\end{array}}
				\norm{H_m  \varphi}_{L^2(\R^3)} ,
\end{split}
\end{equation*}
for $k\in \N^*$ and $m>0$. Here, $\N^* := \N\setminus\{0\}$.
\end{notation}

\subsubsection{The quadratic forms}
At first sight, it might seem surprising that $\lambda_{k}$ and $\lambda_{k,m}$ are related, especially because of the boundary condition of $H^\Omega$. It becomes less surprising when computing the squares of the operators. This is the purpose of the following lemma.

 \begin{lemma}\label{lem:quadf1}
 	Let $\varphi \in \D(H^\Omega)$ and $\psi\in H^1(\R^3; \C^4)$. Then
	\begin{equation}\label{eq:quadMIT}
		\|H^\Omega \varphi\|^2_{L^2(\Omega)} = \mathcal{Q}^{\rm{int}}(\varphi) := \|\nabla \varphi\|^2_{L^2(\Omega)} + \int_{\Gamma}\left(\frac{\kappa}{2}+m_0\right)|\varphi|^2\dx \Gamma + m_0^2\norm{\varphi}_{L^2(\Omega)}^2,
	\end{equation}
	where $\kappa$ is defined in Notation \ref{not.def1},
	and
\begin{equation}\label{eq:quadIM}\begin{split}
		\|H_m \psi\|^2_{L^2(\R^3)} &= 
			\|\nabla\psi\|_{L^2(\Omega)}^2 
			+\|\nabla\psi\|_{L^2(\Omega')}^2 +  \|(m_0 +m\chi_{\Omega'}) \psi\|_{L^2(\R^3)}^2\\
 			&\quad-m\Re\langle\mathcal{B}\psi, \psi\rangle_{\Gamma}\\
			&=\|\nabla\psi\|_{L^2(\Omega)}^2 
			+\|\nabla\psi\|_{L^2(\Omega')}^2 +  \|(m_0 +m\chi_{\Omega'}) \psi\|_{L^2(\R^3)}^2\\
			&\quad+m\|\Xi^-\psi\|^2_{L^2(\Gamma)}-m\|\Xi^+\psi\|^2_{L^2(\Gamma)}\,.
		\end{split}	
\end{equation}	
 \end{lemma}
 \begin{proof}
 The identity \eqref{eq:quadMIT} is proved for instance in \cite[Section A.2]{MR3663620}.
 Let $\psi\in H^1(\R^3; \C^4)$. Then, by integrations by parts,
 \[
 \begin{split}
 \|H_m\psi\|_{L^2(\R^3)}^2 &
 = 
 \|\alpha\cdot D\psi\|_{L^2(\R^3)}^2 +  \|(m_0 +m\chi_{\Omega'}) \psi\|_{L^2(\R^3)}^2
 +2m\Re\langle\alpha\cdot D \psi, \beta \psi\rangle_{\Omega'}
 \\
 &=
  \|\nabla\psi\|_{L^2(\R^3)}^2 +  \|(m_0 +m\chi_{\Omega'}) \psi\|_{L^2(\R^3)}^2
 -m\Re\langle\mathcal{B}\psi, \psi\rangle_{\Gamma}.
 \end{split}
 \]
Then, note that, for all $\psi\in H^1(\R^3;\C^4)$,
\[\Re\langle\mathcal{B}\psi, \psi\rangle_{\Gamma}=\|\Xi^+\psi\|^2_{L^2(\Gamma)}-\|\Xi^-\psi\|^2_{L^2(\Gamma)}\,.\] 

 \end{proof}
Considering \eqref{eq:quadIM} leads to the following minimization problem, for $v\in H^{1}(\Omega)$,
 \begin{equation}\label{pb:extopt}
	\Lambda_m(v) = \inf \{
		\mathcal{Q}_m(u)\ , u\in V_v
	\}\,,\quad \mathcal{Q}_m(u) = \|\nabla u\|^2_{L^2(\Omega')} + m^2\|u\|^2_{L^2(\Omega')}\,,
 \end{equation}
 where 
 \[V_v = \{u\in H^1(\Omega',\mathbb{C}^4) \mbox{ s.t. }u = v\mbox{ on }\Gamma\}\,.\] 
 %
A classical extension theorem (see \cite[Section 5.4]{evans1998partial}) ensures that $V_v$ is non-empty.
 \subsubsection{Heuristics}\label{sec:heuri}
In this paper, we will analyse the behavior of $\Lambda_{m}(v)$ and prove in particular (see Proposition \ref{prop:ext}) that there exists $C>0$ such that for $m$ large, and all $v\in H^1(\Omega;\C^4)$ 
\begin{equation}\label{eq.Lambdam1}
o(1)\geq\Lambda_{m}(v)-\left(m\|v\|^2_{L^2(\Gamma)} + \int_\Gamma \frac{\kappa}{2}|v|^2\dx \Gamma\right)\geq 
 -\frac{C}{m}\|v\|^2_{H^1(\Gamma)}
\,.
\end{equation}
Replacing $m$ by $m_0+m$ in \eqref{eq.Lambdam1}, we get, for all $\psi\in H^1(\R^3;\C^4)$,
\begin{multline}\label{eq.lbHm}
\|H_m\psi\|_{L^2(\R^3)}^2\geq \|\nabla\psi\|^2_{L^2(\Omega)} + m_0^2\norm{\psi}_{L^2(\Omega)}^2
\\
+ \int_\Gamma\left(\frac{\kappa}{2}+m_0\right)|\psi|^2\dx \Gamma+2m\|\Xi^-\psi\|^2_{L^2(\Gamma)} - \frac{C}{m}\|\psi\|^2_{L^2(\Gamma)}\,.
\end{multline}
Take any eigenfunction $\varphi$ of $H^\Omega$ and consider a minimizer $u_{\varphi}$ of \eqref{pb:extopt} for $v=\varphi$ and $m$ replaced by $m+m_0$.
Then, letting $\psi=\mathds{1}_{\Omega} \varphi+\mathds{1}_{\Omega'}u_{\varphi}\in  H^1(\R^3;\C^4)$, 
we get
\[\|H_{m}\psi\|_{L^2(\R^3)}^2
= 
\|\nabla\varphi\|^2_{L^2(\Omega)}
+ m_0^2\norm{\psi}_{L^2(\Omega)}^2
+\Lambda_{m+m_0}(\varphi)-m\|\Xi^+\varphi\|_{L^2(\Gamma)}^2\,.\]
With \eqref{eq.Lambdam1} at hand, we deduce that, for all $j\in\N^*$,
\[\lambda^2_{j,m}\leq \lambda^2_{j}+o(1)\,.\]
Conversely, if we are interested in the eigenvalues of $(H_{m})^2$ that are of order $1$ when $m\to+\infty$, we see from \eqref{eq.lbHm} that the corresponding normalized eigenfunctions must satisfy $\Xi^-\psi=\mathscr{O}(m^{-1})$ and, in particular, $\mathcal{B}\psi=\psi+\mathscr{O}(m^{-1})$. Thus, we get formally, for all $j\in\N^*$,
\[\lambda^2_{j,m}\geq \lambda^2_{j}+o(1)\,.\]
The aim of this paper is to make this heuristics rigorous. We can now state our main theorem.

\begin{theorem}\label{thm:main}
The singular values of $H_{m}$ can be estimated as follows:
	\begin{enumerate}[\rm (i)]
		\item $\lim_{m\to+\infty} \lambda_{k,m} = \lambda_k$, for any $k\in \N^*$.
	
		\item Let $k_1\in \N^*$ be the multiplicity of the first eigenvalue $\lambda_1$ of $|H^\Omega|$. For all $k\in\{1,\dots,k_1\}$, we have
		\[
			\lambda_{k,m} = \left(\lambda_1^2+\frac{\nu_{k}}{m}+o\left(\frac{1}{m}\right)\right)^{1/2},
		\]
		where
		\begin{equation}\label{eq:varmu}
			\nu_{k} = \inf_{\tiny
			\begin{array}{c}
				V\subset \ker(|H^\Omega|-\lambda_1),\\
				\dim V = k,
			\end{array}
			}
			\sup_{\tiny
			\begin{array}{c}
				u\in V,\\
				\norm{u}_{L^2(\Omega)}=1,
			\end{array}
			}
				\eta(u),
		\end{equation}
		with
		\[
			\eta(u) =   
			 	\int_\Gamma	\left(
					\frac{|\nabla_s u|^2}{2}
					-
					\frac{|(\pa_n+\kappa/2+m_0)u|^2}{2}
					+
					\left(\frac{K}{2}-\frac{\kappa^2}{8}-\frac{\lambda_1^2}{2}\right)|u|^2
							\right)\dx\Gamma\,.
		\]
	\end{enumerate}
	Here, $(\lambda_k)_{k\in \N^*}$ and $(\lambda_{k,m})_{k\in \N^*}$ are defined in Notation \ref{not.lambda},
	 $\kappa$ and $K$ are defined in Notation \ref{not.def1}.
	\Bk
\end{theorem}
\begin{remark}
	The max-min formula \eqref{eq:varmu} makes sense since $\ker(|H^\Omega|-\lambda {\rm Id})\subset H^2(\Omega;\C^4)$ for any eigenvalue $\lambda$ of $|H^\Omega|$.
\end{remark}
\begin{remark}
	$H_m$ and $H^\Omega$ anticommute with the charge conjugation $C$ defined for all  $\psi\in \C^4$, by
	\[
		C\psi = i\beta \alpha_2 \overline{\psi}, 
	\]
	where $\overline{\psi}\in \C^4$ is the vector obtained after complex conjugations of each of the components of $\psi$ (see for instance \cite[Section 1.4.6]{Thaller1992} and \cite[Section A.1]{MR3663620}). As a consequence, the spectrum of $H_m$ and $H^\Omega$ are symmetric with respect to $0$ and Theorem \ref{thm:main} may be rewritten as a result on the eigenvalues of  $H_m$ and $H^\Omega$.
\end{remark}
\subsubsection{A vectorial Laplacian with Robin-type boundary conditions}\label{sec.robintro}
Let us also mention an intermediate spectral problem whose study is needed in our proof of Theorem \ref{thm:main} and that may be of interest on its own.
Let us consider the vectorial Laplacian associated with the quadratic form
\begin{equation}\label{eq:quadf2}
	\begin{split}
		\mathcal{Q}_m^{\rm int}(u) = \norm{\nabla u}_{L^2(\Omega)}^2 +m_0^2\norm{u}_{L^2(\Omega)}^2+ \int_\Gamma \left(\frac{\kappa}{2}+m_0\right)|u|^2\dx \Gamma + 2m\norm{\Xi^- u}_{L^2(\Gamma)}^2
	\end{split}
\end{equation}
for $u\in \mathsf{Dom}(\mathcal{Q}_m^{\rm int}) = H^1(\Omega; \C^4)$ and $m>0$ where $\Xi^-, \Xi^+$ are defined by \eqref{eq:boundary_projections}. By a classical trace theorem, this form is bounded from below. More precisely, we have the following result.
\begin{lemma}\label{lem:sarobint}
The self-adjoint operator associated with $\mathcal{Q}_m^{\rm int}$ is defined by
\begin{equation}\label{eq:laplacian_int}
	\begin{split}
	&
	 \mathsf{Dom}(L_m^{\rm int}) = \left\{u\in H^2(\Omega; \C^4) :\quad  \begin{array}{l}
	 	\Xi^-\left(\pa_\n + \kappa/2+m_0 + 2m\right) u =0 \mbox{ on }\Gamma,\\
		\Xi^+\left(\pa_\n + \kappa/2+m_0\right) u = 0 \mbox{ on }\Gamma
			\end{array}\right\}
	\\&
	L_m^{\rm int}u = \left(-\Delta+m_0^2\right) u \mbox{ for all } u \in \mathsf{Dom}(L_m^{\rm int}).
	\end{split}
\end{equation}
It has compact resolvent and its spectrum is discrete. 
\end{lemma}

\begin{notation}\label{not.lambdaint}
Let $(\lambda_{k,m}^{\rm int})_{k\in \N^*}$ denote the sequence of eigenvalues, each one being repeated according to its multiplicity and such that
	\begin{equation}\label{eq:eigen_laplacian_int}
		\lambda_{1,m}^{\rm int}\leq \lambda_{2,m}^{\rm int}\leq \dots
	\end{equation}
\end{notation}
The asymptotic behavior of the eigenvalues of $L_m^{\rm int}$ is detailed in the following theorem.
\begin{theorem}\label{thm:asym_eigen_int-0}
	The following holds:
	\begin{enumerate}[\rm (i)]
		\item For every $k\in \N^*$, $\lim_{m\to +\infty} \lambda_{k,m}^{\rm int} = \lambda_k^2$.
\item 	Let $\lambda$ be an eigenvalue of $|H^\Omega|$ of multiplicity $k_1\in \N^*$. Consider $k_0\in \N$ the unique integer such that for all $k\in\{1,\dots,k_1\}$, $\lambda_{k_0+k} = \lambda$.
 Then, for all $k\in\{1,2,\dots, k_1\}$, we have
		\[
			\lambda_{k_0+k,m}^{\rm int} = \lambda^2 +\frac{\mu_{\lambda,k}}{m} + o\left(\frac{1}{m}\right)
		\]
		where
	\begin{equation}\label{eq:mulk}
			\mu_{\lambda,k} := 
				\inf_{\tiny
			\begin{array}{c}
				V\subset \ker(|H^\Omega|-\lambda),\\
				\dim V = k,
			\end{array}
			}
			\sup_{\tiny
			\begin{array}{c}
				v\in V,\\
				\norm{v}_{L^2(\Omega)}=1,
			\end{array}
			}
			-\frac{
				\norm{(\pa_\n+\kappa/2+m_0)v}_{L^2(\Gamma)}^2
				}{2}\,.
		\end{equation}

	\end{enumerate}
	Here, $(\lambda_k)_{k\in \N^*}$ is defined in Notation \ref{not.lambda}, $(\lambda_{k,m}^{\rm int})_{k\in \N^*}$ in Notation \ref{not.lambdaint} and $\kappa$ in Notation \ref{not.def1}.
	\Bk
\end{theorem}

\subsection{Organization of the paper}
In Section \ref{sec.2}, we discuss the asymptotic properties of the minimizers associated with the exterior optimization problem \eqref{pb:extopt}.
In Section \ref{sec.3}, we investigate the interior problem. In Section \ref{sec.4}, we prove Theorem \ref{thm:main}.

\section{About the exterior optimization problem}\label{sec.2}
The aim of this section is to study the minimizers of \eqref{pb:extopt} and their properties when $m$ tends to $+\infty$. These properties are gathered in the following proposition.

 \begin{proposition}\label{prop:ext}
 For all $v\in H^1(\Omega)$, there exists a unique minimizer $u_m(v)$ associated with $\Lambda_m(v)$, and it satisfies, for all $u\in V_v$,
		\[\mathcal{Q}_m(u) = \Lambda_m(v) + \mathcal{Q}_m(u-u_m(v))\,.\]
There exists a constant $C>0$ such that, for all $m\geq 1$,
	\begin{enumerate}[\rm (i)]
		\item\label{item.i}	 for all $v\in H^1(\Omega)$, we have 
	
		\[o(1)\geq\Lambda_{m}(v)-\left(m\|v\|^2_{L^2(\Gamma)} + \int_\Gamma \frac{\kappa}{2}|v|^2\dx \Gamma\right)
		\geq-\frac{C}{m}\|v\|^2_{L^2(\Gamma)}\,,\]
		and
		\[
			Cm\norm{v}_{H^1(\Omega)}^2\geq\Lambda_{m}(v)\,,
		\]
		\item\label{item.ii} for all $v\in H^2(\Omega)$,
		\[
	 		\left|\Lambda_m(v)-\widetilde\Lambda_m(v)\right|\leq \frac{C}{m^{3/2}}\|v\|^2_{H^{3/2}(\Gamma)},
	 		\]
		\item\label{item.iii}	 for all $v\in H^2(\Omega)$,
		\[
	 		\left|\|u_m(v)\|^2_{L^2(\Omega')} - \frac{\|v\|^2_{L^2(\Gamma)}}{2m}\right|\leq \frac{C}{m^2}\|v\|^2_{H^{3/2}(\Gamma)}\,.
	 	\]
	\end{enumerate}
	Here
	\[
		\widetilde \Lambda_{m}(v) = m\int_\Gamma  
			|v|^2
		\dx\Gamma
		+
		\int_\Gamma  
			\frac{\kappa}{2}|v|^2
		\dx\Gamma
		+
		m^{-1}\int_\Gamma \Big\{
			\frac{|\nabla_s v|^2}{2}
			+
			\Big(\frac{K}{2}-\frac{\kappa^2}{8}\Big)|v|^2
		\Big\}\dx\Gamma.
	\]

 \end{proposition}

\subsection{Existence, uniqueness and Euler-Lagrange equations}

Let us discuss here the existence of the minimizers announced in Proposition \ref{prop:ext} and their elementary properties. We will see later that, in the limit $m\to+\infty$, this minimization problem on $\Omega'$ is closely related to the same problem on a tubular neighborhood in $\Omega'$ of $\Gamma$. For $\delta>0$, $m>0$ and $v\in H^{1}(\Omega)$, we define
	 \begin{equation}\label{pb:extopt1}
	\Lambda_{m,\delta}(v) = \inf \left\{
		 \mathcal{Q}_{m}(u)\;, u\in V_{v,\delta}
	\right\}\,,\qquad  \mathcal{Q}_{m}(u) = \|\nabla u\|^2_{L^2(\mathcal V_{\delta})} + m^2\|u\|^2_{L^2(\mathcal V_{\delta})}\,,
 \end{equation}
 	%
 	where\footnote{Note that, since $\Omega$ is a smooth set, there exists $\delta_0>0$ such that, for all $\delta\in(0,\delta_0)$, the set $\mathcal{V}_\delta$ has the same regularity as $\Omega$.} 
	$\mathcal{Q}_{m}$ is defined in \eqref{pb:extopt},
	 $\mathcal V_{\delta}  = \{\x\in\Omega'~:~{\rm dist}(\x,\Gamma)<\delta\}$
	and
	\[
		V_{v,\delta} =  \left\{
			u\in H^1(\mathcal V_{\delta},\mathbb{C}^4) \mbox{ s.t. }u = v\mbox{ on }\Gamma\mbox{ and }u(\x) =0\mbox{ if }\rm{dist}(\x,\Gamma) = \delta
				\right\}.
	\]
	
	\subsubsection{Existence and uniqueness}

\begin{lemma}\label{lem.existence}
The minimizers associated with \eqref{pb:extopt} and \eqref{pb:extopt1} exist and are unique.
\end{lemma}
\begin{proof}
Let $(u_n)$ and $(u_{\delta,n})$ be minimizing sequences for $\Lambda_{m}(v)$ and $\Lambda_{m,\delta}(v)$ respectively. These two sequences are uniformly bounded in $H^1$ so that up to subsequences, they converge weakly to $u\in H^1(\Omega')$ and $v_\delta\in H^1(\mathcal{V}_\delta)$. By Rellich - Kondrachov compactness Theorem and the interpolation inequality,
 the sequences converges strongly in $H^{s}_{\rm loc}$ for any $s\in[0,1)$. The trace theorem ensures then that the convergence also holds in $L^2_{\rm loc}(\Gamma)$ and $L^2_{\rm loc}(\partial \mathcal{V}_\delta)$ so that $u\in V_v$ and $u_\delta\in V_{v,\delta}$. Since
	\[
		\Lambda_m(v) = \lim_{n\to+\infty} \mathcal{Q}_m(u_n)\geq \mathcal{Q}_m(u)\geq \Lambda_m(v)
	\]                                           
	and
	\[
		\Lambda_{m,\delta}(v) = \lim_{n\to+\infty} \mathcal{Q}_{m}(u_{\delta,n})\geq \mathcal{Q}_{m}(u_{\delta,n})\geq \Lambda_{m,\delta}(v),
	\] 	
	$u$ and $u_\delta$ are minimizers.

	Since $V$ and $V_\delta$ are convex sets and the quadratic form $\mathcal{Q}_m$ is a strictly convex function, the uniqueness follows.

\end{proof}	
	
	\begin{notation}
	The unique minimizers associated with $\Lambda_{m}(v)$ and $\Lambda_{m,\delta}(v)$ are denoted by $u_m(v)$ and $u_{m,\delta}(v)$, respectively, or $u_m$ and $u_{m,\delta}$ when the dependence on $v$ is clear.
	\end{notation}

	\subsubsection{Euler-Lagrange equations}

The following lemma gathers some properties related to the Euler-Lagrange equations.

\begin{lemma}\label{lem:ELE}
	For all $\delta>0$, $m>0$ and $v\in H^1(\Omega)$, the following holds.
	\begin{enumerate}[\rm (i)]
		\item\label{pt:cp1} $(-\Delta+m^2)u_m = 0$ and $(-\Delta+m^2)u_{m,\delta}=0$,
		\item\label{pt:cp2} $\Lambda_m(v) = -\braket{\pa_\n u_m,u_m}_\Gamma$ and $\Lambda_{m,\delta}(v) = -\braket{\pa_\n u_{m,\delta},u_{m,\delta}}_{\Gamma}$,
		\item\label{pt:cp4} $\mathcal{Q}_m(u) =  \Lambda_m(v) + \mathcal{Q}_m(u-u_m)$ for all $u\in V_v$,
		\\
		$\mathcal{Q}_{m}(u) =  \Lambda_{m,\delta}(v) + \mathcal{Q}_{m}(u-u_{m,\delta})$ for all $u\in V_{v,\delta}$,
	\end{enumerate}
	where $\Lambda_m(v)$ and $V_v$ are defined in \eqref{pb:extopt}.
\end{lemma}

\begin{proof}
Let $v\in H^1_0(\Omega')$. The function 
	\[
		\R\ni t\mapsto\mathcal{Q}_m(u_m + tv)
	\]
	has a minimum at $t=0$. Hence, the Euler-Lagrange equation is $(-\Delta+m^2)u_m = 0$. The same proof holds for $u_{m,\delta}$ . The second point follows from integrations by parts.
And for the last point, let $u\in V_v$. We have, by an integration by parts,
	\[\begin{split}
		\mathcal{Q}_m(u-u_m) 
		&= \mathcal{Q}_m(u)+\mathcal{Q}_m(u_m) - 2\Re\braket{u,(-\Delta+m^2)u_m}_{\Omega'} + 2\braket{u_m,\pa_\n u_m}_\Gamma
		\\&
		= \mathcal{Q}_m(u) -\Lambda_m(v)
	\end{split}\]
	and the result follows. The same proof works for $\Lambda_{m,\delta}(v)$.

\end{proof}

\subsection{Agmon estimates}
This section is devoted to the decay properties of the minimizers in the regime $m\to+\infty$.

We will need the following localization formulas.

\begin{lemma}\label{lem.IMS}
 Let $\chi$ be any real bounded Lipschitz function on $\Omega'$.
 Then,
	\begin{equation}\label{eq:idfq1}
		\mathcal{Q}_m(u_m\chi) =  -\braket{\pa_\n u_m,\chi^2u_m}_{\Gamma} + \|(\nabla\chi) u_m \|^2_{L^2(\Omega')}.
	\end{equation}
The same holds for $u_{m,\delta}$.
\end{lemma}

\begin{proof}
By definition, we have
	\[\begin{split}
		&\mathcal{Q}_m(u_m\chi) 
			= 
			m^2\|\chi u_m\|^2_{L^2(\Omega')} + \|(\nabla\chi) u_m + \chi(\nabla u_m)\|^2_{L^2(\Omega')} 
			\\
			&
			=
			m^2\|\chi u_m\|^2_{L^2(\Omega')} 
			+ \|(\nabla\chi) u_m \|^2_{L^2(\Omega')} + \| \chi(\nabla u_m)\|^2_{L^2(\Omega')} 
			+ 2\Re\braket{u_m\chi, \nabla\chi\cdot \nabla u_m}_{\Omega'}\,.
	\end{split}\]
	Then, by an integration by parts,
	\[\begin{split}
		 \| \chi(\nabla u_m)\|^2_{L^2(\Omega')}  = -\braket{\pa_\n u_m,\chi^2u_m}_{\Gamma} -2\Re\braket{u_m\chi, \nabla\chi\cdot \nabla u_m}_{\Omega'} + \Re\braket{-\Delta u_m,\chi^2 u_m}_{\Omega'}\,.
	\end{split}\]
It remains to use Lemma \ref{lem:ELE} to get
\[
	\mathcal{Q}_m(u_m\chi) =  -\braket{\pa_\n u_m,\chi^2u_m}_{\Gamma} + \|(\nabla\chi) u_m \|^2_{L^2(\Omega')}.
\]
The conclusion follows.
\end{proof}

We can now establish the following important proposition.
\begin{proposition}\label{lem:agmon1}
	Let $\gamma\in(0,1)$.
	There exist $\delta_0>0$, and $C_1,\,C_2>0$ such that, for all $\delta\in(0,\delta_0)$ and all $m>0$,
\begin{equation}\label{eq.agmon1a}
\|e^{m\gamma{\rm dist}(\cdot , \Gamma)}u_m\|^2_{L^2(\Omega')}\leq C_1 \|u_m\|^2_{L^2(\Omega')}\,,
\end{equation}
and, for all $v\in V_v$, 
\begin{equation}\label{eq.agmon1b}
(1 - e^{-\gamma m^{1/2}}C_2m^{-1})\Lambda_{m,m^{-1/2}}(v)\leq\Lambda_{m}(v)\leq\Lambda_{m,\delta}(v)\,.
\end{equation}
\
\end{proposition}
		
	\begin{proof}
Let us first prove \eqref{eq.agmon1a}. Given $\eps>0$, we define
	\[
		\Phi : \x \mapsto \min(\gamma \rm{dist}(\x, \Gamma),\eps^{-1}),
	\]
	\[
		\chi_m : \x\mapsto e^{m\Phi(\x)}.
	\]
	Let $ c>1$ and $R>0$.
	Let $\chi_{1,m,R}, \chi_{2,m,R}$ be a smooth quadratic partition of the unity  such that 
	\[
		\chi_{1,m,R}(\x) = \begin{cases}
			1 &\mbox{ if } \mathrm{dist}(\x,\Gamma)\leq R/2m\\
			0 &\mbox{ if } \mathrm{dist}(\x, \Gamma)\geq R/m
		\end{cases}
	\]
	and, for $k\in \{1,2\}$,
	\[
		\|\nabla \chi_{k,m,R}\|_{L^\infty(\Omega')}\leq \frac{2m c}{R}\,.
	\]
	Since $\chi_m$ is a bounded, Lipschitz function and is equal to $1$ on $\Gamma$, we get $u_m\chi_m\in V_v$. 
	By definition and using \eqref{eq:idfq1}, we get
	
	\[\Lambda_m(v) = \mathcal{Q}_m(u_m) = -\braket{\pa_\n u_m,u_m}_{\Gamma} = \mathcal{Q}_m(u_m\chi_m)
		- \|(\nabla\chi_m) u_m \|^2_{L^2(\Omega')}\,.\]
	Then, we use the fact that $\nabla (\chi_{1,m,R}^2+\chi_{2,m,R}^2) = 0$ to get
\[
\begin{split}	
\mathcal{Q}_m(u_m) &=\mathcal{Q}_m(u_m\chi_m\chi_{1,m,R})+\mathcal{Q}_m(u_m\chi_m\chi_{2,m,R})- \|(\nabla\chi_m) u_m \|^2_{L^2(\Omega')}\\&\quad
-\|(\nabla\chi_{1,m,R})\chi_m u_m \|^2_{L^2(\Omega')}
-\|(\nabla\chi_{2,m,R})\chi_m u_m \|^2_{L^2(\Omega')}\,.
\end{split}
\]
	Since 
	$\mathcal{Q}_m(u_m\chi_m\chi_{1,m,R})\geq \Lambda_m(v)$ 
	and 
	\[
	\begin{split}
		\mathcal{Q}_m(u_m\chi_m\chi_{2,m,R})
		&\geq m^2\norm{u_m\chi_m\chi_{2,m,R}}_{L^2(\Omega')}^2 
		\\&
		= m^2\norm{u_m\chi_m}_{L^2(\Omega')}^2 - m^2\norm{u_m\chi_m\chi_{1,m,R}}_{L^2(\Omega')}^2\,,
	\end{split}
	\]
	we get that 
	\begin{multline*}
		 m^2\Big(1-\gamma^2-\frac{8{c}^2}{R^2}\Big)\norm{u_m\chi_m}_{L^2(\Omega')}^2
		 \leq
		 m^2\norm{u_m\chi_m\chi_{1,m,R}}_{L^2(\Omega')}^2 
		 \\
		 \leq 
		 m^2 e^{2m\min\left( \frac{\gamma R}{m}, \frac{1}{\eps}\right)}\norm{u_m}_{L^2(\Omega')}^2 
		 \leq 
		 m^2 e^{2\gamma R}\norm{u_m}_{L^2(\Omega')}^2\,.
	\end{multline*}
	Taking $R>0$ big enough to get $1-\gamma^2-\frac{8{c}^2}{R^2}>0$, we get that
	\[
		\norm{u_m\chi_m}_{L^2(\Omega')}^2\leq C\norm{u_m}_{L^2(\Omega')}^2
	\]
	where $C$ does not depend on $\eps$. Taking the limit $\eps\to0$ and using the Fatou lemma we obtain \eqref{eq.agmon1a}.

Let us now prove \eqref{eq.agmon1b}. We have for any $\delta\in(0,\delta_0)$ that $V_{v,\delta}\subset V_v$ so that 
	\[
	\Lambda_{m}(v)\leq\Lambda_{m,\delta}(v).
	\]
	Let us consider a Lipschitz function $\tilde \chi_m : \Omega' \rightarrow [0,1]$ defined for all $\x\in \Omega'$ by
	\[
		\tilde \chi_m(\x) =
		\begin{cases}
			 1 & \mbox{ if }\mathrm{dist}(\x,\Gamma)\leq \frac{1}{2m^{1/2}}\\
			 0 & \mbox{ if }\mathrm{dist}(\x,\Gamma)\geq \frac{1}{m^{1/2}}
		\end{cases}\,,
	\]
	with $\|\nabla \tilde\chi_{m}\|_{L^\infty(\Omega')}\leq 2 cm^{1/2}$. Thanks to \eqref{eq:idfq1}, we find
\begin{equation}\label{eq.agmon1a bis}	
\Lambda_{m,m^{-1/2}}(v)\leq \mathcal{Q}_m(u_m\tilde \chi_m) = \Lambda_m(v) + \|u_m\nabla \tilde \chi_m\|^2_{L^2(\Omega')}\,.
\end{equation}
Then, by \eqref{eq.agmon1a} we have
\[ \|u_m\nabla \tilde \chi_m\|^2_{L^2(\Omega')}\leq  e^{-\gamma m^{1/2}}4 c^2m\|e^{m\gamma{\rm dist}(\cdot , \Gamma)}u_m\|^2_{L^2(\Omega')}\leq 
C_1e^{-\gamma m^{1/2}}4 c^2m \|u_m\|^2_{L^2(\Omega')}\,.\]
Observing that
\[m\|u_{m}\|^2_{L^2(\Omega')}\leq m^{-1}\Lambda_{m}(v)\,,\]
and using \eqref{eq.agmon1a bis} we easily get \eqref{eq.agmon1b}.	
\end{proof}

\subsection{Optimization problem in a tubular neighborhood}
From Proposition \ref{lem:agmon1}, we see that, in order to estimate $\Lambda_{m}(v)$, it is sufficient to estimate $\Lambda_{m,m^{-1/2}}(v)$. For that purpose, we will use tubular coordinates. 
\subsubsection{Tubular coordinates}
Let $\iota$ be the canonical embedding of $\Gamma$ in $\R^3$ and $g$ the induced metric on $\Gamma$. $(\Gamma,g)$ is a $\mathcal{C}^4$ Riemannian manifold, which we orientate according to the ambient space. Let us introduce the map $\Phi:\Gamma\times(0,\delta)\to\mathcal{V}_{\delta}$ defined by the formula
\begin{equation*}
\Phi(s,t)=\iota(s)+t\n(s)\,
\end{equation*}
where $\mathcal{V}_\delta$ is defined in \eqref{pb:extopt1} below.
The transformation $\Phi$ is a $\mathcal{C}^3$ diffeomorphism for any $\delta\in(0,\delta_0)$ provided that $\delta_0$ is sufficiently small. The induced metric on $\Gamma\times(0,\delta)$ is given by
\[G=g\circ (\mathsf{Id}+tL(s))^{2}+\dx t^2\,,\]
where $L(s)=d\n_{s}$ is the second fundamental form of the boundary at $s$.
Let us now describe how our optimization problem is transformed under the change of coordinates. For all $u\in L^{2}(\mathcal V_{\delta})$, we define the pull-back function
\begin{equation}\label{eq:trans-st}
\widetilde u(s,t):= u(\Phi(s,t)).
\end{equation}
For all $u\in H^{1}(\mathcal{V}_{\delta})$, we have
\begin{equation}\label{eq:bc;n}
\int_{\mathcal{V}_{\delta}}|u|^{2}\dx\x=\int_{\Gamma\times(0,\delta)}|\widetilde u(s,t)|^{2}\,\tilde a\dx \Gamma \dx t\,,
\end{equation}
\begin{equation}\label{eq:bc;qf}
\int_{\mathcal{V}_{\delta}}|\nabla u|^{2}\dx\x= \int_{\Gamma\times(0,\delta)} \Big[\langle\nabla_{s} \widetilde u,\tilde g^{-1}\nabla_{s} \widetilde u\rangle +|\partial_{t}\widetilde u|^{2}\Big]\,\tilde a\dx \Gamma\dx t\,.
\end{equation}
where
\[\tilde g=\big(\mathsf{Id}+tL(s)\big)^2\,,\]
and $\tilde a(s,t)= |\tilde g(s,t)|^{\frac{1}{2}}$. Here $\langle\cdot,\cdot\rangle$ is the Euclidean scalar product and $\nabla_{s}$ is the differential on $\Gamma$ seen through the metric $g$.
Since $L(s)\in \R^{2\times 2}$, we have the exact formula
\begin{equation}\label{eq.Taylor3}
	\tilde{a}(s,t) = 1+t\kappa(s) + t^2K(s)
\end{equation}
%
where $\kappa$  and $K$ are defined in Notation \ref{not.def1}.
\Bk
%
In the following, we assume that 
\begin{equation}\label{eq:depdeltam}
	\delta = m^{-1/2}\,.
\end{equation}
In particular, we will use \eqref{eq:bc;n} and \eqref{eq:bc;qf} with this particular choice of $\delta$.
\subsubsection{The rescaled transition optimization problem in boundary coordinates}\label{sec.rescaled}

We introduce the rescaling
\[(s,\tau)=(s,mt)\,,\]
and the new weights
\begin{equation}\label{eq:Jac-a'}
\widehat a_m(s,\tau)=\tilde a(s,m^{-1}\tau)\,,\qquad \widehat g_m(s,\tau)=\tilde g(s,m^{-1}\tau)\,.
\end{equation}
Note that there exists $m_1\geq1$ such that for all $m\geq m_1$, $s\in \Gamma$ and $\tau\in[0,m^{1/2})$, $\widehat a_m(s,\tau)\geq 1/2$.
\Bk
We set
\begin{equation}\label{eq:dom-hat}
\begin{split}
\widehat{\mathcal V}_{m}&=\Gamma\times(0,\sqrt{m})\,,\\
\widehat V_{m}&=\{u\in H^1(\widehat{\mathcal V}_{m}, \C^4;\widehat a_m \dx \Gamma \dx\tau):~u(\cdot,\sqrt{m})=0\}\,,\\
\widehat{\mathscr{Q}}_{m}(u)&=m^{-1}\int_{\widehat{\mathcal V}_{m}}\Big(\langle\nabla_{s} u,\widehat g^{-1}_m\nabla_{s} u\rangle+m^2|\partial_\tau u|^2\Big)\widehat a_m \dx \Gamma \dx\tau\\
&\quad+m\int_{\widehat{\mathcal V}_{m}}|u|^2\widehat{a}_m\dx \Gamma\dx \tau\,,\\
\widehat{\mathscr{L}}_m&=-m^{-1} \widehat a^{-1}_m\nabla_s(\widehat a_m \widehat g^{-1}_m\nabla_s)
+m\left(-\widehat a^{-1}_m\partial_\tau\widehat a_m\partial_\tau
+1\right)\,.
\end{split}
\end{equation} 
\begin{notation}\label{notnation curvatures}
Let $m\geq m_1$, 
 $k, K\in\R$ and define
\[
	\begin{split}
	a_{m,\kappa,K} :  (0,\sqrt{m})& \longrightarrow \R
	\\
	\tau&\longmapsto1 + \frac{\tau\kappa}{m} + \frac{\tau^2K}{m^2}.
	\end{split}
\]	
\end{notation}
We let
\begin{equation}\label{eq:bcur}
	A =\|\kappa\|_{L^\infty(\Gamma)} \mbox{ and } B = \|K\|_{L^\infty(\Gamma)}\,.
\end{equation}
%
\begin{remark}\label{rem:m0}
We can assume (up to taking a larger $m_1$) that for any
\[
	(m,\kappa,K)\in[m_1,+\infty)\times[-A,A]\times  [-B,B],
\]
we have 
$
	a_{m,\kappa,K}(\tau)\geq 1/2
$
for all $\tau\in(0,\sqrt{m})$.
\end{remark}
In the following, we assume that $m\geq m_1$.

\subsection{One dimensional optimization problem with parameters}
We denote by $\widehat{\mathscr{Q}}_{m,\kappa,K}$ the \enquote{tranversed} quadratic form defined for $u\in H^1((0,\sqrt{m}), a_{m,\kappa,K}\dx\tau)$ by
\[
	\widehat{\mathscr{Q}}_{m,\kappa,K}(u) 
	= \int_{0}^{\sqrt{m}}
		\Big(|\partial_\tau u|^2 + |u|^2\Big)a_{m,\kappa,K} \dx\tau.
\]
We let
\begin{equation}\label{eq:optiextpar}
	\Lambda_{m,\kappa,K} = \inf\{
		\widehat{\mathscr{Q}}_{m,\kappa,K}(u)  : u\in \widehat{V}_{m,\kappa,K}
	\}
\end{equation}
where
\[
	\widehat{V}_{m,\kappa,K} = \left\{
		u \in H^1((0,\sqrt{m}), a_{m,\kappa,K}\dx\tau) : u(0) = 1, \ u(\sqrt{m})=0
	\right\}.
\]
The following lemma follows from the same arguments as for Lemma \ref{lem.existence}.
\begin{lemma}\label{unique minim 1d}
There is a unique minimizer $u_{m,\kappa,K}$ for the optimization problem \eqref{eq:optiextpar}.
\end{lemma}

\begin{lemma}\label{lem.ipp1D}
Let $u,v\in H^1((0,\sqrt{m}), a_{m,\kappa,K}\dx\tau)$ be such that $u(\sqrt{m}) = v(\sqrt{m})= 0$. We have
	\begin{equation}\label{eq:IPPno}\begin{split}
		\int_0^{\sqrt{m}}\langle
		\pa_\tau u
		,\pa_\tau v
	\rangle a_{m,\kappa,K}\dx\tau
	&+
	\int_0^{\sqrt{m}}\langle
		u
		,v	
	\rangle a_{m,\kappa,K}\dx\tau
	\\&
	=
		\int_0^{\sqrt{m}}\Big\langle
		\widehat{\mathcal{L}}_{m,\kappa,K} u
		,v	
	\Big\rangle a_{m,\kappa,K}\dx\tau
	-\langle\pa_\tau u(0),v(0)\rangle\,,
	\end{split}\end{equation}
	where
\[
		\widehat{\mathcal{L}}_{m,\kappa,K} =- a^{-1}_{m,\kappa,K}\partial_\tau a_{m,\kappa,K}\partial_\tau
+1=  -\pa^2_\tau - \frac{m^{-1}\kappa + m^{-2}2K\tau}{1 + m^{-1}\kappa\tau + m^{-2}K\tau^{2}}\pa_\tau + 1\,.
\]
\end{lemma}
\begin{proof}
The lemma follows essentially by integration by parts and Notation \ref{notnation curvatures}.
\end{proof}
\begin{lemma}\label{lem.EL1D}
We have that $u_{m,\kappa,K}\in \mathcal{C}^\infty([0,\sqrt{m}])$ and
\[\widehat{\mathcal{L}}_{m,\kappa,K} u_{m,\kappa,K} = 0\,,\qquad \Lambda_{m,\kappa,K}=-\pa_\tau  u_{m,\kappa,K} (0)\,,\]
where $u_{m,\kappa,K}$ is defined in Lemma \ref{unique minim 1d}.
Moreover, for all $u\in \widehat{V}_{m,\kappa,K}$,
\[\widehat{\mathscr{Q}}_{m,\kappa,K}(u)  = \Lambda_{m,\kappa,K}+\widehat{\mathscr{Q}}_{m,\kappa,K}(u-u_{m,\kappa,K})\,.\]
\end{lemma}
\Bk
\begin{proof}
This follows from Lemma \ref{lem.ipp1D}.

\end{proof}

The aim of this section is to establish an accurate estimate of $\Lambda_{m,\kappa,K}$.
\begin{proposition}\label{lem:ext}
	There exists a constant $C>0$ such that for all 
	\[
	(m,\kappa,K)\in[m_1,+\infty)\times[-A,A]\times  [-B,B],
	\]
\[\left|\Lambda_{m,\kappa,K} - \left(1 + \frac{\kappa}{2m} + \frac{1}{m^2}\left(\frac{K}{2}-\frac{\kappa^2}{8}\right)\right)\right|\leq Cm^{-3}\,,\]
and
\[\left|\int_0^{\sqrt{m}}|u_{m,\kappa,K}|^2a_{m,\kappa,K}\dx \tau - \frac{1}{2}\right| \leq Cm^{-1}\,.\]
\end{proposition}

\begin{proof}
By Lemmas \ref{unique minim 1d} and \ref{lem.EL1D}, the unique solution $u_{m,\kappa,K}$ of the problem satisfies
\[
	\left(
	-\pa_\tau^2
	-\frac{m^{-1}\kappa + m^{-2}2K\tau}{1 + m^{-1}\kappa \tau + m^{-2}K\tau^2}\pa_\tau
	+1
	\right)u_{m,\kappa,K} = 0.
\]
We expand formally $u_{m,\kappa,K}$ as $u_0 + m^{-1}u_1 + m^{-2}u_2 + \mathscr{O}(m^{-3})$: 
\begin{enumerate}[\rm (i)]
\item	For the \emph{zero order} term, we get
\[
	(-\pa_\tau^2 +1)u_0 = 0 \mbox{ and } u_0(1) = 1, \lim_{\tau\to \infty} u_0(\tau) = 0,
\]
so that $u_0(\tau)= e^{-\tau}$.
\item At the \emph{first order},
\[
	(-\pa_\tau^2 +1)u_1 = \kappa\pa_\tau u_0 = -\kappa e^{-\tau} \mbox{ and } u_1(1) = 0, \lim_{\tau\to \infty} u_1(\tau) = 0,
\]
so that $u_1(\tau)= -\frac{\kappa}{2}\tau e^{-\tau}$.
\item At the \emph{second order},
\[\begin{split}
	&(-\pa_\tau^2 +1)u_2 = \kappa\pa_\tau u_1 + (\kappa^2-2K)\tau \pa_\tau u_0 = -\frac{\kappa^2}{2}e^{-\tau} + \left(\frac{3\kappa^2}{2}-2K\right)\tau e^{-\tau},
	\\
	&u_2(0) = 0 \mbox{ and }\lim_{\tau\to \infty} u_2(\tau) = 0,
\end{split}\]
so that
$u_2(\tau)=
	\left(
		\frac{\kappa^2}{8}-\frac{K}{2}
	\right)\tau e^{-\tau}
	+
	\left(
		\frac{3\kappa^2}{8}-\frac{K}{2}
	\right)\tau^2e^{-\tau}.
$
\end{enumerate}
This formal construction leads to define a possible approximation of $u_{m,\kappa,K}$. 
Consider 
\begin{equation}\label{eq:formaltestfun1}\begin{split}
v_{m,\kappa,K}(\tau) &:= \chi_m(\tau)\left(u_0(\tau) + m^{-1}u_1(\tau) + m^{-2}u_2(\tau)\right)\,,\\ 
\chi_m(\tau) &= \chi(\tau/\sqrt{m})\,,
\end{split}\end{equation}
where $\chi : \R_+\mapsto [0,1]$ is a smooth function such that
\[
	\chi(\tau) = \begin{cases}
		1 &\mbox{ if }\tau\in[0,1/2]\\
		0 &\mbox{ if }\tau\geq 1
	\end{cases}\,.
\]
In the following, we denote $v_m = v_{m,\kappa,K}$ to shorten the notation.
\Bk
We immediately get that $v_m$ belongs to $\widehat{V}_{m,\kappa,K}$. Note that
\begin{equation}\label{eq.deriv-vm}
-\partial_{\tau}v_{m}(0)=1 + \frac{\kappa}{2m} + m^{-2}\left(
			\frac{K}{2} - \frac{\kappa^2}{8}
		\right)\,,
\end{equation}
\begin{equation}\label{eq.approx-zero}
\|\widehat{\mathcal{L}}_{m,\kappa,K} v_m\|_{L^2((0,\sqrt{m}), a_{m,\kappa,K}\dx\tau)} = \mathscr{O}(m^{-3})\,.
\end{equation}
Using Lemmas \ref{lem.ipp1D} and \ref{lem.EL1D}, we have
\[	\Lambda_{m,\kappa,K}=
	\int_0^{\sqrt{m}}\Big\langle
		\pa_\tau u_{m,\kappa,K}
		,\pa_\tau v_m
	\Big\rangle a_{m,\kappa,K}\dx\tau
	+
	\int_0^{\sqrt{m}}\Big\langle
		u_{m,\kappa,K}
		,v_m
	\Big\rangle a_{m,\kappa,K}\dx\tau\,,
	\]
and
\[\begin{split}
	&
	 \Lambda_{m,\kappa,K}	
	=\int_0^{\sqrt{m}}\Big\langle
		\widehat{\mathcal{L}}_{m,\kappa,K} v_m	 
		,u_{m,\kappa,K}
	\Big\rangle a_{m,\kappa,K}\dx\tau
	-
	\pa_\tau v_{m}(0)\,.
\end{split}\]
By Lemma \ref{lem.ipp1D}, the Cauchy-Schwarz inequality, \eqref{eq.deriv-vm}, and \eqref{eq.approx-zero},
\[\begin{split}
	\left|\Lambda_{m,\kappa,K}\right.&-\left.\left(
		1 + \frac{\kappa}{2m} + m^{-2}\left(
			\frac{K}{2} - \frac{\kappa^2}{8}
		\right)
	\right)\right|
	\\
	&= \left|
	 	\int_0^{\sqrt{m}}\Big\langle
		\widehat{\mathcal{L}}_{m,\kappa,K} v_m,
		 u_{m,\kappa,K}
	\Big\rangle a_{m,\kappa,K}\dx\tau
	 \right|
	 \\
	 &\leq 
	  \|\widehat{\mathcal{L}}_{m,\kappa,K} v_m\|_{L^2((0,\sqrt{m}), a_{m,\kappa,K}\dx\tau)}
	 \|u_{m,\kappa,K}\|_{L^2((0,\sqrt{m}), a_{m,\kappa,K}\dx\tau)}
	 \\
	 &\leq \Lambda_{m,\kappa, K}^{\frac{1}{2}}  \|\widehat{\mathcal{L}}_{m,\kappa,K} v_m\|_{L^2((0,\sqrt{m}), a_{m,\kappa,K}\dx\tau)} \\
	 &\leq Cm^{-3}\Lambda_{m,\kappa, K}^{\frac{1}{2}}\,.
\end{split}\]
From this, it follows first that $\Lambda_{m,\kappa,K}=\mathscr{O}(1)$ uniformly in $(\kappa, K)$, and then the first estimate of the proposition is established.
Using Lemmas \ref{lem.ipp1D} and \ref{lem.EL1D}, the fact that $v_{m}(0)-u_{m,\kappa,K}(0) = 0$ and Cauchy-Schwarz inequality, we have
\begin{align*}
\widehat{\mathscr{Q}}_{m,\kappa,K}&(v_{m}-u_{m,\kappa,K})\\
&\leq  \|\widehat{\mathcal{L}}_{m,\kappa,K} (v_m-u_{m,\kappa,K})\|_{L^2((0,\sqrt{m}), a_{m,\kappa,K}\dx\tau)}\|v_m-u_{m,\kappa,K}\|_{L^2((0,\sqrt{m}), a_{m,\kappa,K}\dx\tau)}\\
&\leq Cm^{-3}\|v_m-u_{m,\kappa,K}\|_{L^2((0,\sqrt{m}), a_{m,\kappa,K}\dx\tau)}\,.
\end{align*}
The second estimate follows since
\[\|v_m-u_{m,\kappa,K}\|_{L^2((0,\sqrt{m}), a_{m,\kappa,K}\dx\tau)}^2\leq\widehat{\mathscr{Q}}_{m,\kappa,K}(v_{m}-u_{m,\kappa,K})\,,\]
and $\|v_{m}\|_{L^2((0,\sqrt{m}), a_{m,\kappa,K}\dx\tau)}^2=\frac{1}{2}+\mathscr{O}(m^{-1})$.

\end{proof}	
\subsection{Asymptotic study of $\Lambda_{m,m^{-1/2}}(v)$.}
From Proposition \ref{lem:ext} and \eqref{eq:dom-hat}, we deduce the following lower bound.
\begin{corollary}\label{cor:estilowreg}
	There exists $C>0$ such that for any $v\in H^1(\Omega)$,
	\[\begin{split}
		 o(1)&\geq\Lambda_{m,m^{-1/2}}(v)-\left(m\|v\|^2_{L^2(\Gamma)} + \int_\Gamma \frac{\kappa}{2}|v|^2\dx \Gamma\right)
		\geq-\frac{C}{m}\|v\|^2_{L^2(\Gamma)}\,,
	\end{split}\] 
	and
	\[
		Cm\norm{v}_{H^1(\Omega)}^2\geq\Lambda_{m,m^{-1/2}}(v)\,.
	\]
	Here, the term $o(1)$ depends on $v$ (not only on the $H^1$ norm of $v$).
\end{corollary}
\begin{proof}
By Proposition \ref{lem:ext}, the lower bound follows.
By the extension theorem for Sobolev functions (see for instance \cite[Section 5.4.]{evans1998partial}), there exist a constant $C>0$ and, for all $v\in H^1$, a function $Ev\in H^1(\R^3)$ that extends $v$ and such that $\norm{Ev}_{H^1(\R^3)}\leq C\norm{v}_{H^1(\Omega)}$. 
Let us define the test function $u_m$ by $u_m = v \widetilde{u}_m$ where
\[
	\widetilde{u}_m(\Phi(s,t)) = \begin{cases}
		&v_{m,\kappa(s),K(s)}(mt)\mbox{ for all } (s,t)\in\Gamma\times[0,{m}^{-1/2}]\,,\\
		&0\mbox{ for all } (s,t)\in\Gamma\times[{m}^{-1/2},+\infty)\,.
	\end{cases}
\]
Here, the function $v_m$ is defined in \eqref{eq:formaltestfun1}.
With an integration by parts, Lemmas \ref{lem.ipp1D}, \ref{lem.EL1D} and Proposition \ref{lem:ext}, we get
\[\begin{split}
	\mathcal{Q}_m(u_m) 
	&=
	\norm{\widetilde{u}_m\nabla v + v\nabla \widetilde{u}_m}_{L^2(\Omega')}^2 +m^2 \norm{v \widetilde{u}_m}_{L^2(\Omega')}^2
	\\
	&=
	\norm{\widetilde{u}_m\nabla v}_{L^2(\Omega')}^2
	+\norm{ v\nabla \widetilde{u}_m}_{L^2(\Omega')}^2 
	+2\Re\braket{\widetilde{u}_m\nabla v,v\nabla \widetilde{u}_m}_{\Omega'}
	+m^2 \norm{v \widetilde{u}_m}_{L^2(\Omega')}^2
	\\
	&=
	\norm{\widetilde{u}_m\nabla v}_{L^2(\Omega')}^2
	+\braket{\widetilde{u}_mv,v\left(-\Delta+m^2\right) \widetilde{u}_m}_{\Omega'}
	-\braket{v\pa_\n \widetilde{u}_m, v\widetilde{u}_m}_{\Gamma}	
	\\
	&\leq
	\norm{\widetilde{u}_m\nabla v}_{L^2(\Omega')}^2
	+\braket{\widetilde{u}_mv,v\left(-\Delta+m^2\right) \widetilde{u}_m}_{\Omega'}
	\\&+m\norm{v}_{L^2(\Gamma)}^2 + \int_\Gamma|v|^2\kappa/2\dx\Gamma + C\norm{v}_{L^2(\Gamma)}^2/m\,.		
\end{split}\]
Since $\widetilde{u}_m$ is uniformly bounded in $W^{2,\infty}(\Omega)$ and pointwise converges to $0$ with its derivatives in the tangential direction, Lebesgue's dominated convergence theorem ensures that $\norm{\widetilde{u}_m\nabla v}_{L^2(\Omega')}^2$ and $\braket{\widetilde{u}_mv,v\left(-\Delta+m^2\right) \widetilde{u}_m}_{\Omega'}$ tend to $0$ as $m$ goes to $+\infty$. We obtain that
\[
	\limsup_{m\to+\infty}\left(\Lambda_{m,m^{-1/2}}-m\norm{v}_{L^2(\Gamma)}^2 - \int_\Gamma|v|^2\kappa/2\dx\Gamma \right)\leq 0\,.
\]
 \end{proof}

With Proposition \ref{lem:agmon1}, this proves in particular \eqref{item.i} in Proposition \ref{prop:ext}. This section is devoted to the refinement of this lower bound and to the corresponding upper bound.

\subsubsection{Preliminary lemmas}
Let us state a few elementary lemmas that we will use later.

\begin{lemma}\label{lem:Sob.gamma0}
There exists $C>0$ such that, for all $f,g\in H^{\frac{3}{2}}(\Gamma)$, we have
\[\|fg\|_{H^{\frac{3}{2}}(\Gamma)}\leq C\|f\|_{H^{\frac{3}{2}}(\Gamma)}\|g\|_{H^{\frac{3}{2}}(\Gamma)}\,.\]
\end{lemma}
\begin{proof}
$H^{\frac{3}{2}}(\Gamma)$ is an algebra since $\frac{3}{2}>\frac{\dim\Gamma}{2}=1$.
\end{proof}

\begin{lemma}\label{lem:Sob.gamma1}
There exists $C>0$ such that, for all $f\in H^{\frac{3}{2}}(\Gamma)$, we have
\[\|f\|_{H^{\frac{1}{2}}(\Gamma)}\leq C \|f\|^{\frac{1}{2}}_{L^2(\Gamma)}\|f\|^{\frac{1}{2}}_{H^1(\Gamma)}\,.\]
\end{lemma}

\begin{lemma}\label{lem:Sob.gamma2}
There exists $C>0$ such that, for all $f\in H^{\frac{1}{2}}(\Gamma,T\Gamma)$ and $g\in H^{1}(\Gamma,\C)$, we have
\[\left|\int_{\Gamma}f\cdot\nabla_{s} g\dx\Gamma\right|\leq C\|f\|_{H^{\frac{1}{2}}(\Gamma)}\|g\|_{H^{\frac{1}{2}}(\Gamma)}\,.\]
\end{lemma}

\subsubsection{Lower and upper bounds}
\begin{notation}\label{not.vkK}
In the following, we define
\[\begin{split}
	\widehat{\Pi}_m : \,&H^{1}(\Omega, \C^4) \longrightarrow \widehat{V}_m
	\\
	&v \longmapsto [(s,\tau)\in \widehat{\mathcal V}_{m}\mapsto  v(s)u_{m,\kappa(s),K(s)}(\tau)\in \C^4]
\end{split}\]
where $u_{m,\kappa(s),K(s)}$ is defined by Proposition \ref{lem:ext} with $\kappa = \kappa(s) $ and $K = K(s)$.
\end{notation}
\begin{lemma}\label{lem.deriv-u}
We have, uniformly in $s$, 
\[\int_0^{\sqrt{m}}|\nabla_s u_{m,\kappa(\cdot),K(\cdot)}|^2\dx\tau=\mathscr{O}(m^{-2})\,.\]
\end{lemma}
\begin{proof}
We have
\[\left(- a^{-1}_{m,\kappa,K}\partial_\tau a_{m,\kappa,K}\partial_\tau
+1\right)u_{m,\kappa, K}=0\,.\]
Let us take the derivative with respect to $s$:
\[\left(- a^{-1}_{m,\kappa,K}\partial_\tau a_{m,\kappa,K}\partial_\tau
+1\right)\nabla_{s} u_{m,\kappa, K}=\left[\nabla_{s},a^{-1}_{m,\kappa,K}\partial_\tau a_{m,\kappa,K}\partial_\tau\right]u_{m,\kappa, K}\,.\]
Taking the scalar product with $\nabla_{s} u_{m,\kappa, K}$ and integrating by parts by noticing that $\nabla_{s} u_{m,\kappa, K}(0)=0$, we get
\begin{multline*}
\int_{0}^{\sqrt{m}} |\partial_{\tau} \nabla_{s}u_{m,\kappa, K}|^2a_{m,\kappa, K}\dx\tau+\|\nabla_{s} u_{m,\kappa, K}\|^2_{L^2(a_{m,\kappa,K}\dx\tau)}\\
\leq \left|\left\langle\left[\nabla_{s},a^{-1}_{m,\kappa,K}\partial_\tau a_{m,\kappa,K}\partial_\tau\right]u_{m,\kappa, K},\nabla_{s} u_{m,\kappa, K}\right\rangle_{L^2(a_{m,\kappa, K}\dx\tau)}\right|\,.
\end{multline*}
By an explicit computation and the Cauchy-Schwarz inequality, we find
\begin{multline*}
\left|\left\langle\left[\nabla_{s},a^{-1}_{m,\kappa,K}\partial_\tau a_{m,\kappa,K}\partial_\tau\right]u_{m,\kappa, K},\nabla_{s} u_{m,\kappa, K}\right\rangle_{L^2(a_{m,\kappa, K}\dx\tau)}\right|\\
\leq Cm^{-1}\|\partial_{\tau}u_{m,\kappa,K}\|_{L^2(a_{m,\kappa, K}\dx\tau)}\|\nabla_{s}u_{m,\kappa, K}\|_{L^2(a_{m,\kappa, K}\dx\tau)}\,.
\end{multline*}
Since
\[\|\partial_{\tau}u_{m,\kappa,K}\|_{L^2(a_{m,\kappa, K}\dx\tau)}\leq \sqrt{\Lambda_{m,\kappa, K}}\,,\]
we get
\[\int_{0}^{\sqrt{m}} |\partial_{\tau} \nabla_{s}u_{m,\kappa, K}|^2a_{m,\kappa, K}\dx\tau+\|\nabla_{s} u_{m,\kappa, K}\|^2_{L^2(a_{m,\kappa,K}\dx\tau)}\leq Cm^{-2}\,.\]
\end{proof}

\begin{proposition}\label{lem:effop}
	There exist positive constants $C>0$ and $m_1>0$ such that for all $m\geq m_1$, and all $v\in H^{2}(\Omega)$,
	\[
		\big|\Lambda_{m,m^{-1/2}}(v)-\widetilde \Lambda_{m}(v)\Big|\leq Cm^{-3/2}\|v\|^2_{H^{3/2}(\Gamma)}\,,
	\]
	where 
	\[
		\widetilde \Lambda_{m}(v) = m\int_\Gamma  
			|v|^2
		\dx\Gamma
		+
		\int_\Gamma  
			\frac{\kappa}{2}|v|^2
		\dx\Gamma
		+
		m^{-1}\int_\Gamma \left(
			\frac{|\nabla_s v|^2}{2}
			+
			\left(\frac{K}{2}-\frac{\kappa^2}{8}\right)|v|^2
		\right)\dx\Gamma.
	\]
	More precisely, for all $u\in \widehat{V}_m$ such that $u = v$ on $\Gamma$,
\begin{multline*}
		\widehat{\mathscr{Q}}_m(u)\geq \widetilde \Lambda_{m}(v)
		-\frac{C}{m^{3/2}} \|v\|_{H^{3/2}(\Gamma)} ^2
		+\frac{m}{2}
		\|u-\widehat{\Pi}_m v\|^2_{L^2(\widehat{\mathcal{V}}_m, \dx\Gamma\dx\tau)}
		\\
		\qquad\qquad
		+\frac{1}{2m}
		\|\nabla_s \left(u - \widehat{\Pi}_m v\right)\|^2_{L^2(\widehat{\mathcal{V}}_m, \dx\Gamma\dx\tau)}\,,
	\end{multline*}
	and
	\[\begin{split}
		\widehat{\mathscr{Q}}_m(\widehat{\Pi}_m(v))\leq \widetilde \Lambda_{m}(v)
		+Cm^{-3/2} \left(\|v\|^2_{L^2(\Gamma)} + \|\nabla_s v\|^2_{L^2(\Gamma)}\right).
	\end{split}\]
\end{proposition}

\begin{proof}
	Let $v\in H^{2}(\Omega)$. 
	
	First, let us discuss the upper bound. For that purpose, we insert $\widehat{\Pi}_m v$ in the quadratic form:
	\begin{equation*}
	\widehat{\mathscr{Q}}_{m} (\widehat{\Pi}_m v )= 
		m\int_\Gamma  
			\widehat{\mathscr{Q}}_{m,\kappa(\cdot),K(\cdot)}(\widehat{\Pi}_m v )\dx
		\Gamma
		+m^{-1}\int_{\widehat{\mathcal V}_{m}}\langle\nabla_{s} \widehat{\Pi}_m v ,\widehat g^{-1}_m\nabla_{s} \widehat{\Pi}_m v \rangle\widehat a_m \dx \Gamma \dx\tau\,.
\end{equation*}
We have
\[
m\int_\Gamma  \widehat{\mathscr{Q}}_{m,\kappa(\cdot),K(\cdot)}(\widehat{\Pi}_m v )\dx\Gamma
= m\int_\Gamma  |v|^2\Lambda_{m,\kappa(\cdot),K(\cdot)}\dx\Gamma\,,
\]
and
\[\int_{\widehat{\mathcal V}_{m}}\langle\nabla_{s} \widehat{\Pi}_m v ,\widehat g^{-1}_m\nabla_{s} \widehat{\Pi}_m v \rangle\widehat a_m \dx \Gamma \dx\tau
\leq (1+Cm^{-\frac{1}{2}})\int_{\widehat{\mathcal V}_{m}}|\nabla_{s} \widehat{\Pi}_m v|^2 \dx \Gamma \dx\tau\,.
\]
Moreover, for all $\varepsilon>0$,
\begin{multline*}
\int_{\widehat{\mathcal V}_{m}}|\nabla_{s} \widehat{\Pi}_m v|^2 \dx \Gamma \dx\tau\leq(1+\varepsilon)\int_\Gamma
			|\nabla_s v|^2
			\int_0^{\sqrt{m}}
				|u_{m,\kappa(\cdot),K(\cdot)}|^2
			\dx\tau
		\dx\Gamma\\	
	+(1+\varepsilon^{-1})\int_{\Gamma}\|v\|^2\int_0^{\sqrt{m}}|\nabla_{s} u_{m,\kappa(\cdot),K(\cdot)}|^2\dx\tau\dx\Gamma\,.
\end{multline*}
We recall Lemma \ref{lem.deriv-u}. We choose $\varepsilon=m^{-1}$ and recall Proposition \ref{lem:ext} to get
\[\int_{\widehat{\mathcal V}_{m}}|\nabla_{s} \widehat{\Pi}_m v|^2 \dx \Gamma \dx\tau\leq (1+Cm^{-1})\frac{1}{2}\int_{\Gamma}\|\nabla_{s} v\|^2\dx\Gamma+Cm^{-1}\|v\|^2_{L^2(\Gamma)}\,.\]
Therefore,  
\begin{multline*}
\widehat{\mathscr{Q}}_{m} (\widehat{\Pi}_m v )\leq m\int_\Gamma  |v|^2\Lambda_{m,\kappa(\cdot),K(\cdot)}\dx\Gamma+m^{-1}\frac{1+Cm^{-\frac{1}{2}}}{2}\int_{\Gamma}|\nabla_{s} v|^2\dx\Gamma+Cm^{-2}\|v\|^2_{L^2(\Gamma)}\,.
\end{multline*}
It remains to use Proposition \ref{lem:ext} to get the desired upper bound.
Let us now discuss the lower bound. Let $u\in \widehat{V}_m$ such that $u = v \mbox{ on }\Gamma$.
By Lemma \ref{lem.EL1D}, we have\
\[\begin{split}
		\widehat{\mathscr{Q}}_m(u) &= m\int_\Gamma \widehat{\mathscr{Q}}_{m,\kappa(\cdot),K(\cdot)}(u)\dx \Gamma 
		+m^{-1}\int_{\widehat{\mathcal V}_{m}}\langle\nabla_{s} u,\widehat g^{-1}_m\nabla_{s} u\rangle\widehat a_m \dx \Gamma \dx\tau
		\\
		&
		= m\int_\Gamma |v|^2\Lambda_{m,\kappa(\cdot),K(\cdot)}\dx \Gamma
		+m\int_\Gamma \widehat{\mathscr{Q}}_{m,\kappa(\cdot),K(\cdot)}(u-\widehat{\Pi}_m v)\dx \Gamma
		\\
		&
		\qquad+m^{-1}\int_{\widehat{\mathcal V}_{m}}\langle\nabla_{s} u,\widehat g^{-1}_m\nabla_{s} u\rangle\widehat a_m \dx \Gamma \dx\tau\,.
\end{split}\]
Thus,
\begin{multline*}
\widehat{\mathscr{Q}}_m(u)\geq  m\int_\Gamma |v|^2\Lambda_{m,\kappa(\cdot),K(\cdot)}\dx \Gamma+m\left(1-Cm^{-\frac{1}{2}}\right)\|u-\widehat{\Pi}_mv\|^2_{L^2(\widehat{\mathcal{V}}_m, \dx\Gamma\dx\tau)} \\
+m^{-1}\left(1 - Cm^{-\frac{1}{2}}\right)\int_{\widehat{\mathcal V}_{m}}|\nabla_s u|^2 \dx \Gamma \dx\tau\,.
\end{multline*}
\Bk
	We have
	\begin{multline}\label{eq:lemboundext2}
		\int_{\widehat{\mathcal V}_{m}}|\nabla_s u|^2 \dx \Gamma \dx\tau
		=
		\int_{\widehat{\mathcal V}_{m}}|\nabla_s \widehat{\Pi}_mv|^2 \dx \Gamma \dx\tau
		+
		\int_{\widehat{\mathcal V}_{m}}\Big|\nabla_s \left(u - \widehat{\Pi}_mv\right)\Big|^2 \dx \Gamma \dx\tau
		\\
		+ 2\Re \int_{\widehat{\mathcal V}_{m}}\Big\langle\nabla_s \widehat{\Pi}_mv, \nabla_s \left(u - \widehat{\Pi}_mv\right)\Big\rangle \dx \Gamma \dx\tau.
	\end{multline}
	By Lemmas \ref{lem:Sob.gamma2} and \ref{lem:Sob.gamma0},
	\begin{equation*}
		 \left|2\Re \int_{\widehat{\mathcal V}_{m}}\Big\langle\nabla_s \widehat{\Pi}_mv, \nabla_s \left(u - \widehat{\Pi}_mv\right)\Big\rangle \dx \Gamma \dx\tau\right|
		 \leq C
		\|v\|_{H^{3/2}(\Gamma)} 
		 \left\|u - \widehat{\Pi}_mv\right\|_{H^{1/2}(\widehat{\mathcal{V}}_m, \dx\Gamma\dx\tau)}\,.
	\end{equation*}	
	Then, with Lemma \ref{lem:Sob.gamma1}, we get, for all $\varepsilon_0>0$,
	\begin{multline}\label{eq:lemboundext3}
 \Big|2\Re \int_{\widehat{\mathcal V}_{m}}\Big\langle\nabla_s \widehat{\Pi}_mv, \nabla_s \left(u - \widehat{\Pi}_mv\right)\Big\rangle \dx \Gamma \dx\tau\Big|\\
 \leq Cm^{-1}\varepsilon_0^{-1}\|v\|_{H^{3/2}(\Gamma)} ^2
		 + m^{2}\varepsilon_0(1+m^{-2})\left\|u - \widehat{\Pi}_mv\right\|^2_{L^2(\widehat{\mathcal{V}}_m, \dx\Gamma\dx\tau)}
		 +\varepsilon_0\left\|\nabla_s\left(u - \widehat{\Pi}_mv\right)\right\|^2_{L^2(\widehat{\mathcal{V}}_m, \dx\Gamma\dx\tau)}\,.
\end{multline}	
Using \eqref{eq:lemboundext2} and \eqref{eq:lemboundext3}, we get that
	\[\begin{split}
		&
		\widehat{\mathscr{Q}}_m(u)\geq
		m\int_\Gamma  
			|v|^2
		\dx\Gamma
		+
		\int_\Gamma  
			\frac{\kappa}{2}|v|^2
		\dx\Gamma
		+
		m^{-1}\int_\Gamma \left(
			\frac{|\nabla_s v|^2}{2}
			+
			\left(\frac{K}{2}-\frac{\kappa^2}{8}\right)|v|^2
		\right)\dx\Gamma
		\\&
		-\frac{C}{m^{3/2}}(\varepsilon_0^{-1}+1) \|v\|_{H^{3/2}(\Gamma)} ^2
		+m\left(1-\varepsilon_0-\frac{C}{m^{1/2}}\right)
		\|u-\widehat{\Pi}_mv\|^2_{L^2(\widehat{\mathcal{V}}_m, \dx\Gamma\dx\tau)}
		\\
		&
		+m^{-1}\left(1-\varepsilon_0-\frac{C}{m^{1/2}}\right)
		\|\nabla_s \left(u - \widehat{\Pi}_mv\right)\|^2_{L^2(\widehat{\mathcal{V}}_m, \dx\Gamma\dx\tau)}.
	\end{split}\]
	Taking $\varepsilon_0 = 3/4$ and $m$ large enough, we get the result.
	\end{proof}

\subsection{End of the proof of Proposition \ref{prop:ext}}
Item \eqref{item.ii} of Proposition \ref{prop:ext} follows from Propositions \ref{lem:effop} and \ref{lem:agmon1}. It remains to prove \eqref{item.iii}. Consider the minimizer $u_{m}$ and a cut off function $\chi_{m}$ supported in a neighborhood of size $m^{-\frac{1}{2}}$ near the boundary. Then, we let 
\[\check u_{m}(s,\tau)= (\chi_{m} u_{m})\circ \Phi(s,m^{-1}\tau)\,.\]
Let us use the lower bound in Proposition \ref{lem:effop}:
\[
\widehat{\mathscr{Q}}_m(\check u_{m})\geq \widetilde \Lambda_{m}(v)
		+\frac{m}{2}
		\|\check u_{m}-\widehat{\Pi}_m v\|^2_{L^2(\widehat{\mathcal{V}}_m, \dx\Gamma\dx\tau)}-\frac{C}{m^{3/2}} \|v\|_{H^{3/2}(\Gamma)} ^2\,.
\]
As in the proof of Lemma \ref{lem.IMS} and recalling Item \eqref{pt:cp2} in Lemma \ref{lem:ELE}, we get
\[\widehat{\mathscr{Q}}_m(\check u_{m})=\mathcal{Q}_{m, m^{-\frac{1}{2}}}(\chi_{m} u_{m})=\widetilde \Lambda_{m}(v)+\|(\nabla \chi_{m}) u_{m}\|^2=(1+\mathscr{O}(e^{-cm^{\frac{1}{2}}}))\widetilde \Lambda_{m}(v)\,,\]
where we used \eqref{eq.agmon1a}.

We deduce that
\[\|\check u_{m}-\widehat{\Pi}_m v\|^2_{L^2(\widehat{\mathcal{V}}_m, \dx\Gamma\dx\tau)}\leq \frac{C}{m^{5/2}} \|v\|_{H^{3/2}(\Gamma)} ^2\,.\]
Thus
\[\left|\|\check u_{m}\|_{L^2(\widehat{\mathcal{V}}_m, \dx\Gamma\dx\tau)}-\|\widehat{\Pi}_m v\|_{L^2(\widehat{\mathcal{V}}_m, \dx\Gamma\dx\tau)}\right|\leq \frac{C}{m^{5/4}} \|v\|_{H^{3/2}(\Gamma)}\,.\]
Using Proposition \ref{lem:ext}, we get that
\[\left|\|\widehat{\Pi}_m v\|^2_{L^2(\widehat{\mathcal{V}}_m, \dx\Gamma\dx\tau)}-\frac{\|v\|^2_{L^2(\Gamma)}}{2}\right|\leq Cm^{-1}\|v\|^2_{L^2(\Gamma)}\,.\]
Therefore
\[\left|m\|\chi_{m} u_{m}\|^2_{L^2(\mathcal{V}_m, \dx\x)}-\frac{\|v\|^2_{L^2(\Gamma)}}{2}\right|\leq Cm^{-1} \|v\|^2_{H^{3/2}(\Gamma)}\,.\]
We remove $\chi_{m}$ by using \eqref{eq.agmon1a} and Item \eqref{item.iii} follows.

\section{A vectorial Laplacian with Robin-type boundary conditions}\label{sec.3}
In this section, we study the vectorial Laplacian $L_m^{\rm int}$ associated with the quadratic form $\mathcal{Q}_m^{\rm int}$ defined in section \ref{sec.robintro}.

\subsection{Preliminaries :  proof of Lemma \ref{lem:sarobint}}
We recall that the domain of $L_m^{\rm int}$ is the set of the $u\in H^1(\Omega; \C^4)$ such that the linear application
\[ H^1(\Omega; \C^4)\ni v\mapsto \mathcal{Q}_m^{\rm int}(v,u)\in\C\]
is continuous for the $L^2$-norm. By using the Green-Riemann formula, we get that the domain is given by
\[\{u\in H^1(\Omega; \C^4) : -\Delta u\in L^2(\Omega; \C^4)\,,\quad (\pa_\n + \kappa/2+m_0 + 2m\Xi^{-})u=0 \mbox{ on } \Gamma\}\,.\]
By a classical regularity theorem, we deduce that the domain is included in $H^2(\Omega;\C^4)$.
The compactness of the resolvent and the discreteness of the spectrum immediately follow.
\subsection{Asymptotics of the eigenvalues}
In this section, we describe the first terms in the asymptotic expansion of the eigenvalues of $L_m^{\rm int}$. This is the aim of the following proposition.
\begin{proposition}\label{lem:asym_eigen_int}
	The following properties hold.
	\begin{enumerate}[\rm (i)]
		\newcounter{saveenum}
		\item \label{pt:conv_int_asym1} For any $k\in \N^*$, $\lim_{m\to +\infty} \lambda_{k,m}^{\rm int} = \lambda_k^2$.
		\setcounter{saveenum}{\value{enumi}}
	\end{enumerate}
	Let $\lambda$ be an eigenvalue of $|H^\Omega|$ of multiplicity $k_1\in \N^*$. Consider $k_0\in \N$ such that for all $k\in\{1,\dots,k_1\}$, $\lambda_{k_0+k} = \lambda$.
	\begin{enumerate}[\rm (i)]
		\setcounter{enumi}{\value{saveenum}}
		\item\label{pt:conv_int_asym2} For all $k\in\{1,2,\dots, k_1\}$, we have
		\[
			\lambda_{k_0+k,m}^{\rm int} = \lambda^2 +\frac{\mu_{\lambda,k}}{m} + o\left(\frac{1}{m}\right)
		\]
		where
	\begin{equation}
			\mu_{\lambda,k} := 
				\inf_{\tiny
			\begin{array}{c}
				V\subset \ker(|H^\Omega|-\lambda),\\
				\dim V = k,
			\end{array}
			}
			\sup_{\tiny
			\begin{array}{c}
				v\in V,\\
				\norm{v}_{L^2(\Omega)}=1,
			\end{array}
			}
			-\frac{
				\norm{(\pa_\n+\kappa/2+m_0)v}_{L^2(\Gamma)}^2
				}{2}\,.
		\end{equation}
		\item	\label{pt:conv_int_asym3}Let $(u_{k_0+1},\dots,u_{k_0+k_1})$ be a $H^1$-weak limit of a sequence \[(u_{k_0+1,m},\dots,u_{k_0+k_1,m})_{m>0}\] of $L^2$-orthonormal eigenvectors of $L^{\rm int}_m$ associated with the eigenvalues \[(\lambda_{k_0+1,m}^{\rm int},\dots,\lambda_{k_0+k_1,m}^{\rm int})\,.\] 
		Then, we have for all $v\in\ker(|H^\Omega|-\lambda)$ that,
		\[-\frac{1}{2}\norm{(\pa_\n+\kappa/2+m_0)v}_{L^2(\Gamma)}^2=\sum_{k=1}^{k_1}|\braket{ v, u_{k_0+k}}_\Omega|^2\mu_{\lambda,k}\,.\]
	\end{enumerate}

	Here, $(\lambda_k)_{k\in \N^*}$ is defined in Notation \ref{not.lambda} and $(\lambda_{k,m}^{\rm int})_{k\in \N^*}$ in Notation \ref{not.lambdaint}.
\end{proposition}

\subsection{Proof of Proposition \ref{lem:asym_eigen_int}}

	Since $\mathsf{Dom}(H^\Omega)\subset \mathsf{Dom}(\mathcal{Q}_{m}^{\rm int})$, we have
	\begin{equation}\label{eq:upper_bound_op_int}
		\lambda_k^2\geq \lambda_{k,m}^{\rm int}
	\end{equation}
	for all $k\in \N^*$ and all $m>0$.
	
\subsubsection{Lower bounds}
\begin{lemma}\label{lem:induc1}
Let $k\in\N$.
The following properties hold:
	\begin{enumerate}[\rm (i)]
		\item\label{pt:ind_int_asymp1} For all $j\in \{1,2,\dots, k\}$, $\lim_{m\to +\infty} \lambda_{j,m}^{\rm int} = \lambda_j^2$.
		\item\label{pt:ind_int_asymp2} For all subsequence $(m_n)_{n\in \N^*}$ going to $+\infty$ as $n\to+\infty$, all $L^2$-orthonormal family of eigenvectors $(u_{1,m_n}, \dots, u_{k,m_n})$ of $L_{m_n}^{\rm int}$ associated with $(\lambda_{1,m_n}^{\rm int},\dots, \lambda_{k,m_n}^{\rm int})$ such that the sequence $(u_{1,m_n}, \dots, u_{k,m_n})_{n\in \N^*}$ converges weakly in $H^1$, then the sequence $(u_{1,m_n}, \dots, u_{k,m_n})_{n\in \N^*}$ converges strongly in $H^1$ and 
		\begin{equation}\label{eq.limbnd}
			\lim_{n\to+\infty}m_{n}\norm{\Xi^-u_{j,m_n}}_{L^2(\Gamma)}^2 = 0
		\end{equation}
		for all $j\in\{1,\dots,k\}$.
	\end{enumerate} 
\end{lemma}
\begin{proof}
Let us prove \eqref{pt:ind_int_asymp1} and \eqref{pt:ind_int_asymp2} by induction on $k\in \N^*$.
\paragraph{Case $k=0$} There is nothing to prove.

\paragraph{Case $k> 0$} Assume that \eqref{pt:ind_int_asymp1} and \eqref{pt:ind_int_asymp2} are valid for some $k\in \N$. 

Let $(u_{1,m}, \dots, u_{k+1,m})$ be an $L^2$-orthonormal family of eigenvectors  of $L_{m}^{\rm int}$ associated with $(\lambda_{1,m}^{\rm int},\dots, \lambda_{k+1,m}^{\rm int})$. By \eqref{eq:upper_bound_op_int} and the trace Theorem \cite[Section 5.5]{evans1998partial}, the sequence $(u_{1,m}, \dots, u_{k+1,m})_{m>0}$ is bounded in $H^1(\Omega; \C^4)^{k+1}$, and 
	\begin{equation}\label{eq:ineq_asymp_eigen_int}
		\lambda^2_{k+1}\geq \limsup_{m\to+\infty}\lambda_{k+1,m}^{\rm int}\geq \liminf_{m\to+\infty}\lambda_{k+1,m}^{\rm int}.
	\end{equation}
	Hence there exists a subsequence $(m_n)_{n\in \N^*}$ going to $+\infty$ as $n\to+\infty$ such that
	\[
		\lim_{n\to+\infty}\lambda_{k+1,m_n}^{\rm int} = \liminf_{m\to+\infty}\lambda_{k+1,m}^{\rm int}
	\]
	and $(u_{1,m_n}, \dots, u_{k+1,m_n})_{n\in \N^*}$ converges weakly in $H^1(\Omega;\C^4)$ to $(u_1,\dots, u_{k+1})$.

	Using the induction assumption, we get that $(u_{1,m_n}, \dots, u_{k,m_n})_{n\in \N^*}$ converges strongly in $H^1(\Omega;\C^4)$ to $(u_1,\dots, u_{k})$, $\lim_{m\to +\infty} \lambda_{j,m}^{\rm int} = \lambda_j^2$ and	
		\[
			\lim_{n\to+\infty}m\norm{\Xi^-u_{j,m_n}}_{L^2(\Gamma)}^2 = 0
		\]
		for all $j\in\{1,\dots,k\}$. 
		By Rellich-Kondrachov Theorem \cite[Section 5.7]{evans1998partial}, the sequence $(u_{k+1,m_n})$ converges strongly in $L^2(\Omega;\C^4)$. This implies that $(u_1, \dots ,u_{k+1})$ is an $L^2$-orthonormal family. In addition, for all $j_1,j_2\in \{1,\dots, k+1\}$, $j_1\ne j_2$, and all $n\in \N^*$,
		\begin{multline*}
			0 = \Re\braket{\nabla u_{j_1,m_n},\nabla u_{j_2,m_n}}_\Omega+m_0^2\Re\braket{u_{j_1,m_n},u_{j_2,m_n}}_\Omega 
			\\
			+ \Re\braket{(\kappa/2+m_0) u_{j_1,m_n}, u_{j_2,m_n}}_\Gamma + 2m_{n}\Re\braket{\Xi^- u_{j_1,m_n}, \Xi^-u_{j_2,m_n}}_\Gamma
		\end{multline*}
		and taking the limit $n\to+\infty$, 
		\[
			0 = \Re\braket{\nabla u_{j_1},\nabla u_{j_2}}_\Omega +m_0^2\Re\braket{u_{j_1},u_{j_2}}_\Omega+ \Re\braket{(\kappa/2 +m_0)u_{j_1}, u_{j_2}}_\Gamma.
		\]	
		Since
		\[
			 \lim_{n\to+\infty} \mathcal{Q}_{m_n}^{\rm int}(u_{j,m_n}) = \lambda^2_j = \mathcal{Q}^{\rm int}(u_{j})
		\]
		for all $j\in\{1,\dots,k\}$, where $\mathcal{Q}^{\rm int}$ is defined in \eqref{eq:quadMIT}, we deduce that the $(u_{j})_{1\leq j\leq k}$ are normalized eigenfunctions associated with $(\lambda^2_{j})_{1\leq j\leq k}$.  By the min-max theorem, we get
		\[
			 \liminf_{n\to+\infty} \mathcal{Q}_{m_n}^{\rm int}(u_{k+1,m_n})\geq \mathcal{Q}^{\rm int}(u_{k+1})
				\geq \lambda^2_{k+1}\,.
		\]
		We deduce that 
		\[\lim_{m\to+\infty}\lambda_{k+1,m}^{\rm int} = \lambda^2_{k+1}\,.\] 
		We also get that
		\[
			\lim_{n\to+\infty}\norm{\nabla u_{k+1,m_n}}_{L^2(\Omega)} = \norm{\nabla u_{k+1}}_{L^2(\Omega)}
		\]
		and the strong convergence follows. Note that $\lim_{m\to+\infty}\lambda_{k+1,m}^{\rm int} = \lambda^2_{k+1}$ implies that the previous arguments are valid for any weakly converging subsequence and Items \eqref{pt:ind_int_asymp1} and \eqref{pt:ind_int_asymp2} follow for $k+1$.
\end{proof}

\subsubsection{A technical lemma} 
The following lemma is essential in the proof of Items \eqref{pt:conv_int_asym2} and \eqref{pt:conv_int_asym3}.
\begin{lemma}\label{lem:asympt_intlem1}
Let $k\in \N^*$ and $m>0$.
		Let $u$ \emph{resp.} $u_{k,m}$ be a $L^2$-normalized eigenfunction of $|H^\Omega|$  \emph{resp.} $L_m^{\rm int}$ associated with the eigenvalues $\lambda$  \emph{resp.} $\lambda_{k,m}^{\rm int}$.  Then
			\begin{equation}\label{eq:iden_asym_int}
			m(\lambda_{k,m}^{\rm int}-\lambda^2)\braket{u_{k,m},u}_{\Omega} = -1/2\braket{(\pa_\n + \kappa/2+m_0)u_{k,m},(\pa_\n + \kappa/2+m_0) u}_\Gamma\,.
		\end{equation}
		\end{lemma}
		
		\begin{proof}
		Since
		\[
			\begin{array}{ll}
				\Xi^+ (\pa_\n + \kappa/2+m_0)u = 0,  &\Xi^- u = 0,\\
				\Xi^+ (\pa_\n + \kappa/2+m_0)u_{k,m} = 0,& \Xi^-(\pa_\n + \kappa/2+m_0+2m)u_{k,m} = 0
			\end{array}\mbox{ on }\Gamma\,,
		\]
		an integration by parts gives
		\[\begin{split}
			&(\lambda_{k,m}^{\rm int}-\lambda^2)\braket{u_{k,m},u}_{\Omega} 
			= \braket{(-\Delta+m_0^2) u_{k,m},u}_\Omega - \braket{u_{k,m},(-\Delta+m_0^2) u}_\Omega 
			\\&= -\braket{\pa_\n u_{k,m},u}_\Gamma +\braket{u_{k,m},\pa_\n u}_\Gamma
			\\&= -\braket{(\pa_\n+\kappa/2+m_0) u_{k,m},u}_\Gamma +\braket{u_{k,m},(\pa_\n + \kappa/2+m_0) u}_\Gamma
			\\&= \braket{\Xi^-u_{k,m},\Xi^-(\pa_\n + \kappa/2+m_0) u}_\Gamma
			\\&= -1/2m\braket{\Xi^-(\pa_\n + \kappa/2+m_0)u_{k,m},\Xi^-(\pa_\n + \kappa/2+m_0) u}_\Gamma\,.
		\end{split}\] 
		\end{proof}
	
	\subsubsection{Proof of Items \eqref{pt:conv_int_asym2} and \eqref{pt:conv_int_asym3}}
		Let also $(u_{1,m_n},\dots, u_{k_0+k_1,m_n})_{n\in \N^*}$ be a sequence of $L^2$-orthonormal eigenvectors of $L_{m_n}^{\rm int}$ that converges strongly in $H^1(\Omega; \C^4)^{k_0+k_1}$ to an $L^2$-orthonormal family $(u_1,\dots, u_{k_0+k_1})$ of eigenvectors of $|H^\Omega|$. 
		We have
		\[
			{\rm span}(u_{k_0+1}, \dots, u_{k_0+k_1}) = \ker(|H^\Omega|-\lambda)\,.
		\]
		By \eqref{eq:iden_asym_int}, we have
		for all $v=\sum_{k=1}^{k_{1}} a_{k}u_{k_{0}+k}$,
		\[\begin{split}
			&-1/2\norm{(\pa_\n+\kappa/2+m_0)v}_{L^2(\Gamma)}^2 = \sum_{k,j=1}^{k_1}\overline{a_k}a_j\braket{(\pa_\n+\kappa/2+m_0)u_{k_0+k}, (\pa_\n+\kappa/2+m_0)u_{k_0+j}}_\Gamma
			\\&
			=\lim_{n\to+\infty}\sum_{k,j=1}^{k_1}\overline{a_k}a_j\braket{(\pa_\n+\kappa/2+m_0)u_{k_0+k,m_n}, (\pa_\n+\kappa/2+m_0)u_{k_0+j}}_\Gamma
			\\&
			=\lim_{n\to+\infty}\sum_{k,j=1}^{k_1}\overline{a_k}a_jm_n(\lambda^{\rm int }_{k_0+k,m_n}-\lambda^2)\braket{u_{k_0+k,m_n},u_{k_0+j}}_\Omega
			\\&
			=\lim_{n\to+\infty}\sum_{k,j=1}^{k_1}\overline{a_k}a_jm_n(\lambda^{\rm int}_{k_0+k,m_n}-\lambda^2)\braket{u_{k_0+k},u_{k_0+j}}_\Omega
			\\&
			=\lim_{n\to+\infty}\sum_{k=1}^{k_1}|a_k|^2m_n(\lambda^{\rm int}_{k_0+k,m_n}-\lambda^2)\,.
		\end{split}\]
		We deduce that for all $k\in\{1,\dots,k_1\}$,
		\begin{multline*}
			\lim_{n\to+\infty}m_n(\lambda^{\rm int}_{k_0+k,m_n}-\lambda^2)
			=
			-\frac{
				\norm{(\pa_\n+\kappa/2+m_0)u_{k_0+k}}_{L^2(\Gamma)}^2 
				}{2}
			\\
			= \inf_{\tiny
			\begin{array}{c}
				V\subset \ker(|H^\Omega|-\lambda),\\
				\dim V = k,
			\end{array}
			}
			\sup_{\tiny
			\begin{array}{c}
				v\in V,\\
				\norm{v}_{L^2(\Omega)}=1,
			\end{array}
			}
			-\frac{
				\norm{(\pa_\n+\kappa/2+m_0)v}_{L^2(\Gamma)}^2
				}{2} = \mu_{\lambda,k}\,,
		\end{multline*}
		so that
		\[
			\lim_{m\to+\infty}m(\lambda^{\rm int}_{k_0+k,m}-\lambda^2) = \mu_{\lambda_,k}\,.
		\]
		The conclusion follows.

\section{Proof of the main theorem}\label{sec.4}
\subsection{First term in the asymptotic}
In this part, we work in the energy space without using any regularity result as Lemma \ref{lem:regu}. 
\subsubsection{Upper bound}
Let $K\in\N^*$ and $( \varphi_{1},\dots, \varphi_{K})$ be an $L^2$-orthonormal family of eigenvectors of $|H^\Omega|$ associated with the eigenvalues $(\lambda_{1},\dots,\lambda_{K})$.  Using Proposition \ref{prop:ext}, we extend these functions outside $\Omega$ by
\[
	\widetilde u_{j,m} = \begin{cases}
		u_{j}& \mbox{ on } \Omega\,,\\
		u_{m+m_0}(u_{j}) &\mbox{ on } \Omega'\,,
		\end{cases}
\]
for $j\in\{1,\dots,K\}$. By Proposition \ref{prop:ext}, we get that
\[
	\norm{\widetilde u_{j,m}}_{L^2(\Omega')}^2\leq (m+m_0)^{-2}\Lambda_{m+m_0}(u_j)\leq \frac{C}{m+m_0}\,.
\]
so that $\widetilde u_{1,m},\dots,\widetilde u_{K,m}$ are linearly independent vectors.
Let $a_1,\dots,a_K\in\C$. Let us denote $\varphi_{m}^a := \sum_{j=1}^Ka_j\widetilde{u}_{j,m}$.
By Lemma \ref{lem:quadf1} and Proposition \ref{prop:ext}, we have
\[\begin{split}
	&\norm{H_m\varphi_m^a}_{L^2(\R^3)}^2 
	= \norm{\nabla\varphi_m^a}_{L^2(\Omega)}^2 + m_0^2\norm{\varphi_m^a}_{L^2(\Omega)}^2
	-m\Re\langle\mathcal{B}\psi, \psi\rangle_{\Gamma} + \Lambda_{m+m_0}(\varphi_m^a)
	\\
	&\leq
	\mathcal{Q}^{\rm int}\left(\sum_{j=1}^Ka_ju_j\right)+o(1) = \sum_{j=1}^K|a_j|^2\lambda_j^2+o(1)\leq \lambda_K^2 \sum_{j=1}^K|a_j|^2+o(1)\,.
\end{split}\]
We deduce that
\begin{equation}\label{eq:bsfirsttot}
	\limsup_{m\to+\infty}\lambda_{K,m}^2
	\leq \limsup_{m\to+\infty}
		\sup_{\tiny\begin{array}{c}
			\varphi_m^a\in{\rm span}(\widetilde{u}_{1,m},\dots,\widetilde{u}_{K,m})\,
			\\
			\norm{\varphi_m^a}_{L^2(\R^3)}=1\,
			\end{array}
		}\norm{H_m\varphi_m^a}_{L^2(\R^3)}^2\leq \lambda_K^2\,.
\end{equation}
\subsubsection{Lower bound and convergence}
Let $K\in\N^*$ and $( \varphi_{1,m},\dots, \varphi_{K,m})$ be an $L^2$-orthonormal family of eigenvectors of $|H_m|$ associated with the eigenvalues $(\lambda_{1,m},\dots,\lambda_{K,m})$ for all $m\geq m_1$. 
By \eqref{eq:bsfirsttot}, there exists $C>0$ such that
\begin{equation}\label{eq:unifbound}
	C\geq \sup_{\tiny\begin{array}{c}k\in\{1,\dots,K\},\\ m\geq m_1,\end{array}}\norm{H_m  \varphi_{k,m}}_{L^2(\R^3)}^2\,.
\end{equation}
%
%
%
%
%
With \eqref{eq:quadIM} and Proposition \ref{prop:ext}, we get, for all $k\in\{1,\dots,K\}$ and all $m\geq m_1$, that
\begin{equation}\label{eq:firstest}\begin{split}
	&
	\lambda_{k,m}^2 = \|H_{m} \varphi_{k,m}\|_{L^2(\R^3)}^2
	\\
	&
	=
	\norm{\nabla  \varphi_{k,m}}_{L^2(\Omega)}^2 + m_0^2\norm{ \varphi_{k,m}}_{L^2(\Omega)}^2
	-m\braket{\mathcal{B} \varphi_{k,m}, \varphi_{k,m}}_\Gamma
	\\
	&\qquad+
	\Lambda_{m+m_0}( \varphi_{k,m})
	+\mathcal{Q}_{m+m_0}( \varphi_{k,m} - u_{m+m_0}( \varphi_{k,m}))
	\\
	&
	\geq \mathcal{Q}^{\rm int}_m( \varphi_{k,m}) 
	+(m+m_0)^2\norm{ \varphi_{k,m} - u_{m+m_0}( \varphi_{k,m})}_{L^2(\Omega')}^2
	-\frac{C}{m}\norm{ \varphi_{k,m}}_{L^2(\Gamma)}^2.
\end{split}\end{equation}
By the trace theorem, we deduce that there exists $C>0$ such that
\begin{equation}\label{eq:estih1}
	C\geq  \sup_{\tiny\begin{array}{c}k\in\{1,\dots,K\},\\ m\geq m_1,\end{array}} \norm{ \varphi_{k,m}}_{H^1(\Omega)}\,.
\end{equation}
Note also that by \eqref{eq:firstest}, \eqref{eq:estih1} and the trace theorem, we get that 
\begin{equation}\label{eq:estil2tot}
	\left|\norm{ \varphi_{k,m}}_{L^2(\Omega')}-\norm{u_{m+m_0}( \varphi_{k,m})}_{L^2(\Omega')}\right|\leq \norm{ \varphi_{k,m} - u_{m+m_0}( \varphi_{k,m})}_{L^2(\Omega')}\leq C/m\,.
\end{equation}
Moreover, by Proposition \ref{prop:ext}, we obtain that
\[
	\norm{u_{m+m_0}( \varphi_{k,m})}_{L^2(\Omega')}^2
	\leq (m+m_0)^{-2}\Lambda_{m+m_0}( \varphi_{k,m})\leq C(m+m_0)^{-1}\norm{ \varphi_{k,m}}_{H^1(\Omega)}^2\,,
\]
and we deduce that
\begin{equation}\label{eq:estiextnorm}
	\norm{ \varphi_{k,m}}_{L^2(\Omega')}\leq Cm^{-1}.
\end{equation}
Combining \eqref{eq:firstest}, \eqref{eq:estih1}, \eqref{eq:estiextnorm}, Proposition \ref{lem:asym_eigen_int} with an induction procedure as in the proof of Lemma \ref{lem:induc1}, we get the following result.

\begin{lemma}\label{lem:induc2}
Let $K\in\N$.
The following properties hold.
	\begin{enumerate}[\rm (i)]
		\item\label{pt:ind_tot_asymp1} For all $j\in \{1,2,\dots, K\}$, $\lim_{m\to +\infty} \lambda_{j,m} = \lambda_j$.
		\item\label{pt:ind_tot_asymp2} For all subsequence $(m_n)_{n\in \N}$ going to $+\infty$ as $n\to+\infty$, all $L^2$-orthonormal family of eigenvectors $( \varphi_{1,m_n}, \dots,  \varphi_{K,m_n})$ of $|H_m|$ associated with $(\lambda_{1,m_n},\dots, \lambda_{K,m_n})$ such that the sequence $( \varphi_{1,m_n}, \dots,  \varphi_{K,m_n})_{n\in \N}$ converges weakly in $H^1(\Omega)$, then the sequence $( \varphi_{1,m_n}, \dots,  \varphi_{K,m_n})_{n\in \N}$ converges strongly in $H^1(\Omega)$ and 
		\begin{equation}\label{eq.lim_tot_bnd}
			\lim_{n\to+\infty}m_{n}\norm{\Xi^- \varphi_{j,m_n}}_{L^2(\Gamma)}^2 = 0
		\end{equation}
		for all $j\in\{1,\dots,K\}$.
		\item\label{pt:ind_tot_asymp3} Any weak limit $(\varphi_1,\dots,\varphi_K)$ of such a sequence is an $L^2$-orthonormal family of eigenvectors of $|H^\Omega|$ associated with the eigenvalues $(\lambda_1,\dots,\lambda_K)$.
	\end{enumerate} 
\end{lemma}

\subsection{Second term in the asymptotic}
In this section, we will freely use the following regularity result.
\begin{lemma}\label{lem:regu}
	There exists a constant $C>0$ such that for any $m\in \R$, any eigenfunction $u$ of $H_m$ associated with an eigenvalue $\lambda\in \R$, we have 
	\[
		\|u\|_{H^2(\Omega)}\leq C(1+|\lambda|)\|u\|_{L^2(\R^3)}.
	\] 
	We also have, for any eigenfunction $u$ resp. $v$ of $H^\Omega$ resp. $L^{\rm int}_m$ associated with an eigenvalue $\lambda\in \R$, resp. $\lambda^2\in \R$ that
	\[
		\|u\|_{H^2(\Omega)}\leq C(1+|\lambda|)\|u\|_{L^2(\Omega)}
	\] 
	and
	\[
		\|v\|_{H^2(\Omega)}\leq C(1+|\lambda|)\|v\|_{L^2(\Omega)}
	\] 
\end{lemma}

\subsubsection{Upper bound}
In this section, we prove the following lemma.
\begin{lemma}
Let $\lambda$ be an eigenvalue of $|H^\Omega|$ of multiplicity $k_1\in \N^*$. Let $k_0\in \N$ be the unique integer such that  
\[
	\lambda = \lambda_{k_0+1} = \dots = \lambda_{k_0+k_1}.
\]
We have
\begin{equation}\label{eq:bsconclusion}
	\limsup_{m\to+\infty} m(\lambda_{k_0+k,m}^2-\lambda^2)\leq\widetilde\nu_{\lambda,k}\,.
\end{equation}
where
\begin{equation}\label{eq:bsminmax} 
	\widetilde \nu_{\lambda,k}
	:=
	\inf_{\tiny
	\begin{array}{c}
		V\subset \ker(|H^\Omega|-\lambda \rm{Id})\,,\\
		\dim V = k\,,
	\end{array}
	}
	\sup_{\tiny
	\begin{array}{c}
		v\in V\,,\\
		\norm{v}_{L^2(\Omega)}=1\,,
	\end{array}
	}
	 \widetilde \eta_\lambda(v)\,,
\end{equation}
and
\[
\widetilde \eta_\lambda(v) := \int_\Gamma \left(\frac{|\nabla_s v|^2}{2} -\frac{|(\pa_\n+\kappa/2+m_0)v|^2}{2}+ \left(\frac{K}{2} - \frac{\kappa^2}{8}-\frac{\lambda^2}{2}\right)|v|^2\right)\dx \Gamma\,,
\]
for $k\in \{1,\dots,k_1\}$.
\end{lemma}
\begin{proof}
Let $(u_{1,m}, \dots,u_{k_0+k_1,m})$ be an $L^2$-orthonormal family of eigenvectors of $L^{\rm int}_m$ associated with the eigenvalues $(\lambda^{\rm int}_{1,m},\dots,\lambda^{\rm int}_{k_0+k_1,m})$. 
Let $(m_n)_{n\in \N}$ be a subsequence which goes to $+\infty$ as $n$ tends to $+\infty$ and which satisfies
\begin{enumerate}[\rm (i)]
	\item $\limsup_{m\to+\infty}m(\lambda_{k_0+k,m}^2-\lambda^2) = \lim_{n\to+\infty}m_n(\lambda_{k_0+k,m_n}^2-\lambda^2)$,
	\item $(u_{1,m_n},\dots,u_{k_0+k_1,m_n})$ converges in $L^2(\Omega)$ to $(u_1,\dots,u_{k_0+k_1})$,
\end{enumerate}
where $(u_1,\dots,u_{k_0+k_1})$ is an $L^2$-orthonormal family  of eigenvectors of $H^\Omega$ associated with the eigenvalues $(\lambda_1,\dots,\lambda_{k_0+k_1})$.
By Lemma \ref{lem:regu}, this sequence is uniformly bounded in $H^2(\Omega)$. By interpolation, the convergence also holds in $H^s(\Omega)$ for all $s\in[0,2)$.
Since \eqref{eq:bsminmax} is a finite dimensional spectral problem, there exists an $L^2$-orthonormal basis $(w_{k_0+1},\dots,w_{k_0+k_1})$ of $\ker(|H^\Omega|-\lambda{\rm Id})$ such that
\[
	\widetilde{\eta}_\lambda\left(\sum_{s=k_0+1}^{k_0+k_1}a_sw_s\right) = \sum_{s=k_0+1}^{k_0+k_1}|a_s|^2\widetilde{\eta}_\lambda(w_s) =  \sum_{s=k_0+1}^{k_0+k_1}|a_s|^2\widetilde \nu_{\lambda,s-k_0}\,,
\]
for all $a_{k_0+1},\dots,a_{k_0+k_1}\in \C$.
Moreover, we have
\[
	\ker(|H^\Omega|-\lambda{\rm Id}) = {\rm span}(u_{k_0+1},\dots u_{k_0+k_1}) = {\rm span}(w_{k_0+1},\dots w_{k_0+k_1})\,,
\]
so that there exists a unitary matrix $B\in \C^{k_1\times k_1}$ such that $Bu = w$ where $u = (u_{k_0+1},\dots,u_{k_0+k_1})^T$ and $w = (w_{k_0+1},\dots,w_{k_0+k_1})^T$.
Using Proposition \ref{prop:ext}, we extend these functions outside $\Omega$ by
\[
	\widetilde u_{j,m} = \begin{cases}
		u_{j,m}& \mbox{ on } \Omega\,,\\
		u_{m+m_0}(u_{j,m}) &\mbox{ on } \Omega'\,,
		\end{cases}
\]
for $j\in\{1,\dots,k_0+k_1\}$.
We also define 
\[\begin{split}
	&u_m := (u_{k_0+1,m},\dots,u_{k_0+k_1,m})^T\\
	&w_{m} = (w_{k_0+1,m},\dots,w_{k_0+k_1,m})^T := Bu_m\\
	&\widetilde w_{m} = (\widetilde w_{k_0+1,m},\dots,\widetilde w_{k_0+k_1,m})^T := B(\widetilde u_{k_0+1,m},\dots,\widetilde u_{k_0+k_1,m})^T\,,
\end{split}\]
and
\[\begin{split}
	&V_{k_0+k,m} = {\rm span}(u_{1,m},\dots u_{k_0,m}, w_{k_0+1,m},\dots, w_{k_0+k})\,,
	\\
	&\widetilde{V}_{k_0+k,m} = {\rm span}( \widetilde{u}_{1,m},\dots  \widetilde{u}_{k_0,m},  \widetilde{w}_{k_0+1,m},\dots,  \widetilde{w}_{k_0+k})\,,
\end{split}\]
for all $k\in\{k_0+1,\dots,k_0+k_1\}$ and all $m\geq m_1$. Let us remark that
\[
	\dim V_{k_0+k,m} = \dim \widetilde{V}_{k_0+k,m} = k_0+k
\]
for all $k\in\{1,\dots,k_1\}$  (choosing if necessary a larger constant $m_1>0$). 
In the following, we consider test functions of the form
\[
	v_m = \sum_{j=1}^{k_0}a_j\widetilde u_{j,m}+\sum_{j=k_0+1}^{k_0+k_1}a_j\widetilde w_{j,m}\,,
\]
where $a_{1},\dots, a_{k_0+k_1}\in\C$ satisfies $\sum_{j=1}^{k_0+k_1}|a_j|^2 = 1$ so that 
\[
	\|v_m\|_{L^2(\Omega)}^2 = \sum_{j=1}^{k_0+k_1}|a_j|^2 = 1\,.
\]
 By Proposition \ref{prop:ext}, we have 
 \begin{equation}\label{eq:bs2}
 	\norm{v_m}_{L^2(\R^3)}^2 = \norm{v_m}_{L^2(\Omega)}^2 + \norm{v_m}_{L^2(\Omega')}^2
	= 1 + \frac{\norm{v_m}_{L^2(\Gamma)}^2}{2m} + \mathscr{O}(m^{-2}) \,,
 \end{equation}
 and
 \begin{equation}\label{eq:bs3}\begin{split}
 	&\norm{H_m v_m}_{L^2(\R^3)}^2 = \mathcal{Q}^{\rm int}_{m}(v_m) + m^{-1}\int_\Gamma \left(\frac{|\nabla_s v_m|^2}{2} + \left(\frac{K}{2} - \frac{\kappa^2}{8}\right)|v_m|^2\right)\dx \Gamma + \mathscr{O}(m^{-3/2})\,.
	\end{split}
 \end{equation}
From \eqref{eq:bs2} and \eqref{eq:bs3}, we deduce that
\begin{multline*}
m\left(\frac{\norm{H_m v_m}_{L^2(\R^3)}^2}{\norm{v_m}_{L^2(\R^3)}^2}-\lambda^2\right)
\leq m\left(\mathcal{Q}^{\rm int}_m(v_m)-\lambda^2\right)
\\
+ \int_\Gamma \left(\frac{|\nabla_s v_m|^2}{2} + \left(\frac{K}{2} - \frac{\kappa^2}{8}- \frac{\mathcal{Q}_m^{\rm int}(v_m)}{2}\right)|v_m|^2\right)\dx \Gamma + \mathscr{O}(m^{-1/2})\,.
\end{multline*}
For $k\in\{1,\dots,k_1\}$, we get
\begin{equation}\label{eq:bs1}\begin{split}
	&
	m\left(\lambda_{k_0+k,m}^2-\lambda^2\right)
	\\&
	\leq
	\sup_{\tiny
	\begin{array}{c}
		v_m\in \widetilde{V}_{k_0+k,m}\setminus \{0\}\,,
	\end{array}
	}
	m\left(\frac{\norm{H_m v_m}_{L^2(\R^3)}^2}{\norm{v_m}_{L^2(\R^3)}^2}-\lambda^2\right)
	\\&
	\leq
	\sup_{\tiny
	\begin{array}{c}
		v_m\in {V}_{k_0+k,m}\,,\\
		\norm{v_m}_{L^2(\Omega)}=1
	\end{array}
	}
	m\left(\mathcal{Q}^{\rm int}_m(v_m)-\lambda^2\right) + \overline{\eta}_m(v_m) + \mathscr{O}(m^{-1/2})\,,
\end{split}\end{equation}
where
\[
	\overline{\eta}_m(v) := \int_\Gamma \left(\frac{|\nabla_s v|^2}{2} + \left(\frac{K}{2} - \frac{\kappa^2}{8}-\frac{\mathcal{Q}_m^{\rm int}(v)}{2}\right)|v|^2\right)\dx \Gamma\,.
\]
The remaining of the proof concerns the asymptotic behavior of
\[
	\mu_{k,m} := \sup_{\tiny
	\begin{array}{c}
		v_m\in {V}_{k_0+k,m}\,,\\
		\norm{v_m}_{L^2(\Omega)}=1
	\end{array}
	}
	m\left(\mathcal{Q}^{\rm int}_m(v_m)-\lambda^2\right) + \overline{\eta}_m(v_m)\,,
\]
for $k\in\{1,\dots, k_1\}$ when $m$ goes to $+\infty$. Let us first remark that for any $v_m\in V_{k_0+k,m}$, we have
\[\begin{split}
	&v_m = \sum_{j=1}^{k_0}a_j u_{j,m}+\sum_{j=k_0+1}^{k_0+k}a_j w_{j,m} = \sum_{j=1}^{k_0}a_j u_{j,m}+\sum_{s=k_0+1}^{k_0+k_1}\left(\sum_{j=k_0+1}^{k_0+k}a_jb_{j,s} \right)u_{s,m}\,,
\end{split}\]
where $(b_{j,s})_{j,s\in \{k_0+1,\dots,k_0+k_1\}} = B$.
With Proposition \ref{lem:asym_eigen_int}, we obtain
\begin{equation}\label{eq:bs4}\begin{split}
	&
	m_n\left(\mathcal{Q}^{\rm int}_{m_n}(v_{m_n})-\lambda^2\right)
	\\&
	= \sum_{j=1}^{k_0}m_n(\lambda_{j,m_n}^{\rm int}-\lambda^2)|a_j|^2
	+
	\sum_{j=k_0+1}^{k_0+k_1}m_n(\lambda_{j,m_n}^{\rm int}-\lambda^2)\left|\sum_{s=k_0+1}^{k_0+k}a_sb_{s,j} \right|^2
	\\
	&
	=\sum_{j=1}^{k_0}m_n(\lambda_{j,m_n}^{\rm int}-\lambda^2)|a_j|^2
	-\frac{\norm{\left(\pa_n+\kappa/2+m_0\right)\sum_{j=k_0+1}^{k_0+k}a_j w_{j}}_{L^2(\Gamma)}^2}{2}
	+o(1)\,.
\end{split}\end{equation}
Using \eqref{eq:bs1} and \eqref{eq:bs4} and taking $a_1= \dots= a_{k_0+k-1} = 0$, $a_{k_0+k}= 1$, we deduce that
\begin{equation}\label{eq:bs5}
	\liminf_{n\to+\infty}\mu_{k,m_n}\geq \widetilde \nu_{\lambda,k}\,.
\end{equation}
Let $(v^n)_{n\in \N}$ be a sequence of maximizer of $\mu_{k,m_n}$. For all $n$, there exists a unitary vector $a^n =(a_{1,n},\dots, a_{k_0+k,n})\in \C^{k_0+k} $ such that
\[
	v^n = \sum_{j=1}^{k_0}a_{j,n}u_{j,m_n} + \sum_{j=k_0+1}^{k_0+k}a_{j,n}w_{j,m_n}\,.
\]
Up to a subsequence, we can assume that $(a^n)$ converges in $\C^{k_0+k}$ to a unitary vector $a = (a_{k_0+1},\dots,a_{k_0+k})$.
Proposition \ref{lem:asym_eigen_int}, \eqref{eq:bs4} and \eqref{eq:bs5} ensure that 
\[
	\lim_{n\to+\infty}\lambda_{j,m_n}^{\rm int}-\lambda^2 \leq \lambda_j^2-\lambda^2<0
\] for $j\in\{1,\dots,k_0\}$ so that there exists $c_0>0$ such that
\[
	m_n\sum_{j=1}^{k_0}|a_{j,n}|^2\leq c_0\,
\]
and
\[
	\limsup_{n\to+\infty}\mu_{k,m_n}\leq \widetilde \eta_\lambda(v)\leq \widetilde \nu_{\lambda,k}
\]
where $v = \sum_{j=k_0+1}^{k_0+k}a_jw_j$. With \eqref{eq:bs1}, we conclude noticing that $\lim_{n\to+\infty}\mu_{k,m_n}=\widetilde \nu_{\lambda,k}$ and
\[\limsup_{m\to+\infty} m(\lambda_{k_0+k,m}^2-\lambda^2)\leq\widetilde\nu_{\lambda,k}\,.
\]
\end{proof}
\subsubsection{Lower bound}
In the following, we look for the second term in the asymptotic expansions of the eigenvalues. More precisely, we will show the following lemma.
\begin{lemma}
	We have for all $k\in \{1,\dots,k_1\}$ that
	\[
	\liminf_{m\to+\infty}m(\lambda_{k,m}^2-\lambda_1^2)
	\geq  \widetilde \nu_{\lambda_1,j}\,,
\]
where $\widetilde \nu_{\lambda_1,j}$ is defined in \eqref{eq:bsminmax}.
\end{lemma}
\begin{proof}
Let $\lambda$ be the first eigenvalue of $|H^\Omega|$ whose multiplicity is denoted $k_1\in \N^*$: 
\[
	\lambda = \lambda_{1} = \dots = \lambda_{k_1}.
\]
By Lemma \ref{lem:induc2} and Proposition \ref{lem:asym_eigen_int}, we have
\[
	\lim_{m\to+\infty}\lambda_{k,m}^2  = \lim_{m\to+\infty}\lambda^{\rm int}_{k,m} = \lambda^2\,,
\]
for all $k\in \{1,\dots,k_1\}$.
Let $(\varphi_{1,m},\dots,\varphi_{k_1,m})$ an $L^2$-orthonormal family of eigenvectors of $|H_m|$ associated with the eigenvalues $(\lambda_{1,m},\dots,\lambda_{k_1,m})$ for all $m\geq m_1$.
By Lemma \ref{lem:regu}, there exists $C>0$ such that
\begin{equation}\label{eq:unifboundH2int}
	C\geq \sup_{\tiny\begin{array}{c}
	m\geq m_1,
	\\
	j\in\{1,\dots,k_1\},
	\end{array}}
	\norm{\varphi_{j,m}}_{H^2(\Omega)}.
\end{equation}
Let us remark that for all $k\in \{1,\dots,k_1\}$, and all $m\geq m_1$,
\[\begin{split}
	&
	\lambda_{k,m}^2 
	=
	\norm{H_m \varphi_{k,m}}_{L^2(\R^3)}^2
	=
	\sup_{
		\tiny
		\begin{array}{c}
			(a_{1},\dots,a_{k})\in\C^{k},
			\\
			\sum_{j = 1}^{k}|a_j|^2 = 1,
		\end{array}
	} 
	\norm{H_m \left(\sum_{j=1}^{k}a_j\varphi_{j,m}\right)}_{L^2(\R^3)}^2\,.
\end{split}\]
Let $a = (a_{1},\dots,a_{k})\in\C^{k}$ be such that $\sum_{j = 1}^{k}|a_j|^2 = 1$. We define
\[
	\varphi_{m}^a = \sum_{j=1}^{k}a_j\varphi_{j,m}.
\]
With \eqref{eq:quadIM}, \eqref{eq:unifboundH2int} and Proposition \ref{prop:ext}, we get
\begin{equation}\label{eq:secest}\begin{split}
	&
	\lambda_{k,m}^2 
	\geq \mathcal{Q}^{\rm int}_m\left(\varphi_{m}^a\right) 
	+ m^{-1}\int_\Gamma \left(\frac{|\nabla_s \varphi_{m}^a|^2}{2} + \left(\frac{K}{2} - \frac{\kappa^2}{8}\right)|\varphi_{m}^a|^2\right)\dx \Gamma
	\\
	&\qquad
	+(m+m_0)^2\norm{\varphi_{m}^a - u_{m+m_0}(\varphi_{m}^a)}_{L^2(\Omega')}^2
	+\mathscr{O}(m^{-3/2}).
\end{split}\end{equation}
By \eqref{eq:estil2tot}, we get
\[\begin{split}
	&\left|\norm{\varphi_{m}^a}_{L^2(\Omega')}^2-\norm{u_{m+m_0}(\varphi_{m}^a)}_{L^2(\Omega')}^2\right|
	\\&\qquad
	\leq 
	C/m\left(\norm{\varphi_{m}^a}_{L^2(\Omega')}+\norm{u_{m+m_0}(\varphi_{m}^a)}_{L^2(\Omega')}\right)
	\\
	&\qquad
	\leq C/m\left(\norm{\varphi_{m}^a - u_{m+m_0}(\varphi_{m}^a)}_{L^2(\Omega')}+2\norm{u_{m+m_0}( \varphi_{m}^a)}_{L^2(\Omega')}\right)
	\\
	&\qquad
	\leq C/m\left(m^{-1}+2\norm{u_{m+m_0}( \varphi_{m}^a)}_{L^2(\Omega')}\right)\,.
\end{split}\]
Using Proposition  \ref{prop:ext} and \eqref{eq:unifboundH2int}, we deduce that
\[
	\left|
		\norm{u_{m+m_0}(\varphi_{m}^a)}_{L^2(\Omega')}^2 - \frac{\norm{\varphi_{m}^a}_{L^2(\Gamma)}^2}{2m}
	\right|
	\leq \frac{C}{m^{3/2}}\,,
\]
so that
\begin{equation}\label{eq:normL2ext}
	\left|\norm{\varphi_{m}^a}_{L^2(\Omega')}^2- \frac{\norm{\varphi_{m}^a}_{L^2(\Gamma)}^2}{2m}\right|\leq \frac {C}{m^{3/2}}.
\end{equation}
With \eqref{eq:secest} and Proposition \ref{lem:asym_eigen_int}, we obtain
\begin{equation}\label{eq:thiest}\begin{split}
	&
	m(\lambda_{k,m}^2-\lambda^2) 
	\\
	&
	\geq m\left(\mathcal{Q}^{\rm int}_m(\varphi_{m}^a)-\lambda^2\norm{\varphi_{m}^a}_{L^2(\Omega)}^2\right) 
	\\
	&\qquad+ \int_\Gamma \left(\frac{|\nabla_s\varphi_{m}^a|^2}{2} + \left(\frac{K}{2} - \frac{\kappa^2}{8}-\frac{\lambda^2}{2}\right)|\varphi_{m}^a|^2\right)\dx \Gamma
	+\mathscr{O}(m^{-1/2})\,.
\end{split}\end{equation}
Let $(u_{j,m})_{j\in \N^*}$ be an $L^2$-orthonormal basis of $L^2(\Omega;\C^4)$ whose elements are eigenvectors of $L^{\rm int}_m$ associated with the sequence of eigenvalues $(\lambda_{j,m}^{\rm int})$.
Since $\lambda_{j,m}^{\rm int}$ converges to $\lambda_j^2$ as $m$ goes to $+\infty$, we get that
\[
	\lambda_{j,m}^{\rm int}-\lambda^2\geq 0,
\]
for all $j\geq k_1+1$ and all $m\geq m_1$ (choosing if necessary a larger constant $m_1>0$).
We deduce that
\begin{equation}\label{eq:limlb2}\begin{split}
	&
	m\left(\mathcal{Q}^{\rm int}_m(\varphi_{m}^a)-\lambda^2\norm{\varphi_{m}^a}_{L^2(\Omega)}^2\right)
	=
	\sum_{s=1}^{+\infty}
		m\left(\lambda_{s,m}^{\rm int}-\lambda^2\right)|\braket{\varphi_m^a, u_{s,m}}_{\Omega}|^2
	\\&
	\geq
	\sum_{s=1}^{k_1}
		m\left(\lambda_{s,m}^{\rm int}-\lambda^2\right)|\braket{\varphi_m^a, u_{s,m}}_{\Omega}|^2\,.	
\end{split}\end{equation}

Let $(m_n)_{n\in \N^*}$ be a subsequence which goes to $+\infty$ as $n$ tends to $+\infty$ and such that
\begin{enumerate}[\rm (i)]
	\item $\liminf_{m\to+\infty}m(\lambda_{k,m}^2-\lambda^2) = \lim_{n\to+\infty}m_n(\lambda_{k,m_n}^2-\lambda^2)$,
	\item $(u_{1,m_n},\dots,u_{k_1,m_n})$ converges in $H^1(\Omega)$ to $(u_1,\dots,u_{k_1})$,
	\item $(\varphi_{1,m_n},\dots,\varphi_{k_1,m_n})$ converges in $H^1(\Omega)$ to $(\varphi_1,\dots,\varphi_{k_1})$, \end{enumerate}
where $(u_1,\dots,u_{k_1})$ and $(\varphi_1,\dots,\varphi_{k_1})$ are $L^2$-orthonormal families  of eigenvectors of $H^\Omega$ associated with the eigenvalue $\lambda$.
By Proposition \ref{lem:asym_eigen_int}, we have that
\begin{equation}\label{eq:limlbfin}\begin{split}
	&\lim_{n\to+\infty}\sum_{s=1}^{k_1}
		m\left(\lambda_{s,m_n}^{\rm int}-\lambda^2\right)|\braket{\varphi_{m_n}^a, u_{s,m_n}}_{\Omega}|^2
	\\&
	\quad
	=
	\sum_{s=1}^{k_1}
		-\frac{\norm{(\pa_\n+\kappa/2+m_0)u_s}_{L^2(\Gamma)}^2}{2}|\braket{\varphi^a, u_{s}}_{\Omega}|^2
	=
	-\frac{\norm{(\pa_\n+\kappa/2+m_0)\varphi_a}_{L^2(\Gamma)}^2}{2}\,,
\end{split}\end{equation}
where $\varphi^a = \sum_{j=1}^{k}a_j\varphi_j$.
We get from \eqref{eq:thiest}, \eqref{eq:limlb2}, and \eqref{eq:limlbfin} that
\[
	\liminf_{m\to+\infty}m(\lambda_{k,m}^2-\lambda^2)
	\geq \widetilde \eta_\lambda(\varphi^a)\,,
\]
and
\[
	\liminf_{m\to+\infty}m(\lambda_{k,m}^2-\lambda^2)
	\geq
	\sup_{
		\tiny
		\begin{array}{c}
			(a_{1},\dots,a_k)\in \C^{k},
			\\
			\sum_{j=1}^{k}|a_j|^2 = 1,
		\end{array}
		}
		 \widetilde \eta_\lambda(\varphi^a)
	\geq  \widetilde \nu_{\lambda,j}\,.
\]
The conclusion follows from the upper bound \eqref{eq:bsconclusion}.
\end{proof}
\begin{remark}
	When considering a larger eigenvalue $\lambda > \lambda_{1}$, the proof above breaks down since
	\[
		\sum_{s=1}^{k_0}
		m\left(\lambda_{s,m}^{\rm int}-\lambda^2\right)|\braket{\varphi_m^a, u_{s,m}}_{\Omega}|^2
	\]
	is non positive and the non-wanted terms in \eqref{eq:limlb2} cannot be removed so easily anymore. Here $k_0$ denotes the unique integer such that
	\[
		\lambda = \lambda_{k_0+1} = \dots = \lambda_{k_0+k_1}.
	\] 
\end{remark}
\appendix
\section{Sketch of the proof of Lemma \ref{lem:regu}}
The purpose of this appendix is to give the main ideas of the proof of Lemma \ref{lem:regu}. We do not intend to give a rigorous proof but rather to enlighten why the classical arguments give uniform bounds in $m$  (see for instance \cite[Section 6.3]{evans1998partial}). In particular, we restrict ourselves to the operator $H_m$ for $\Omega : \R^3_+ = \{\x = (x_1,x_2,x_3) : x_3>0\}$ and consider the solution $u\in H^1(\R^3; \C^4)$ of 
\[
	H_m u = (\alpha\cdot D +(m_0+m\chi_{\R^3_-})\beta)u  = f\,,
\]
where $f\in H^1(\R^3;\C^4)$.
By Lemma \ref{lem:quadf1} and Proposition \ref{lem:ext}, we have
\begin{multline*}
	\norm{f}_{L^2(\R^3)}^2\geq \left(\norm{\nabla u}_{L^2(\Omega)}^2 +m_0^2\norm{u}_{L^2(\Omega)}^2 +m_0\norm{u}_{L^2(\Gamma)}^2+ \sum_{k=1}^2\norm{\pa_k u}_{L^2(\Omega')}^2 \right) \\
	+ 2m\norm{\Xi^- u}_{L^2(\Gamma)}^2 - C/m\norm{u}_{L^2(\Gamma)}^2\,,
\end{multline*}
so that by the trace theorem, there exists $C>0$ such that
\begin{equation}\label{lem.reguesti1}\begin{split}
	&C\left(\norm{f}_{L^2(\R^3)}^2+\norm{u}_{L^2(\Omega)}^2\right)\geq \norm{\nabla u}_{L^2(\Omega)}^2  + \sum_{k=1}^2\norm{\pa_k u}_{L^2(\Omega')}^2 \,.
\end{split}\end{equation}
Using the notation of \cite[Section 6.3]{evans1998partial}, we introduce the difference quotients
\[
	D^h_ku(\x) = \frac{u(\x+he_k)-u(\x)}{h}, \quad h\in \R, h\ne 0, \x\in\R^3, k\in \{1,2,3\}\,.
\]
For $j\in \{1,2\}$, we get that
\[
	H_m D^h_j u = (\alpha\cdot D +(m_0+m\chi_{\R^3_-})\beta)D^h_j u  =D^h_j f\,,
\]
so that using \eqref{lem.reguesti1}, we get
\begin{equation*}\begin{split}
	&C\left(\norm{D^h_j f}_{L^2(\R^3)}^2+\norm{D^h_j u}_{L^2(\Omega)}^2\right)\geq \norm{\nabla D^h_ju}_{L^2(\Omega)}^2  + \sum_{k=1}^2\norm{\pa_k D^h_j u}_{L^2(\Omega')}^2 \,.
\end{split}\end{equation*}
By \cite[Section 5.8.2]{evans1998partial}, we deduce that
\begin{equation}\label{lem.reguesti2}\begin{split}
	&C\left(\norm{\pa_j f}_{L^2(\R^3)}^2+\norm{\pa_j u}_{L^2(\Omega)}^2+\norm{ f}_{L^2(\R^3)}^2+\norm{ u}_{L^2(\Omega)}^2\right)\\&\qquad\geq \norm{\nabla \pa_ju}_{L^2(\Omega)}^2  + \sum_{k=1}^2\norm{\pa_k \pa_j u}_{L^2(\Omega')}^2 \,.
\end{split}\end{equation}
We also have that on $\Omega$,
\[
	-\pa_3^2u = H_m^2u+(\sum_{k=1}^2\pa_k^2-m_0^2)u = H_mf\,,
\]
so that
\begin{equation}\label{lem.reguesti3}
	\norm{\pa_3^2 u}_{L^2(\Omega)}\leq C\norm{f}_{H^1(\Omega)}.
\end{equation}
Using \eqref{lem.reguesti1}, \eqref{lem.reguesti2} and \eqref{lem.reguesti3}, we get the result.

\subsection*{Fundings}
N.~Arrizabalaga was partially supported by ERCEA Advanced Grant 669689-HADE, MTM2014-53145-P (MICINN, Gobierno de Espa\~na) and IT641-13 (DEUI, Gobierno Vasco). 
L.~Le Treust was partially supported by ANR DYRAQ ANR-17-CE40-0016-01. 
A.~Mas was partially supported by
MTM2017-84214 and MTM2017-83499 projects of the MCINN (Spain),
2017-SGR-358 project of the AGAUR (Catalunya), and ERC-2014-ADG project 
HADE Id.\! 669689 (European Research Council).
\bibliographystyle{abbrv}      
\bibliography{Bib_MIT}

\begin{thebibliography}{10}

\bibitem{PhysRevB.77.085423}
A.~R. Akhmerov and C.~W.~J. Beenakker.
\newblock Boundary conditions for dirac fermions on a terminated honeycomb
  lattice.
\newblock {\em Phys. Rev. B}, 77:085423, Feb 2008.

\bibitem{MR3663620}
N.~Arrizabalaga, L.~Le~Treust, and N.~Raymond.
\newblock On the {MIT} {B}ag {M}odel in the {N}on-relativistic {L}imit.
\newblock {\em Comm. Math. Phys.}, 354(2):641--669, 2017.

\bibitem{MR901725}
M.~V. Berry and R.~J. Mondragon.
\newblock Neutrino billiards: time-reversal symmetry-breaking without magnetic
  fields.
\newblock {\em Proc. Roy. Soc. London Ser. A}, 412(1842):53--74, 1987.

\bibitem{bogoliubov1987}
P.~Bogolioubov.
\newblock Sur un mod{\`e}le {\`a} quarks quasi-ind{\'e}pendants.
\newblock {\em Annales de l'I.H.P., section A}, 8:163--189, 1968.

\bibitem{MIT061974}
A.~Chodos, R.~L. Jaffe, K.~Johnson, C.~B. Thorn, and V.~F. Weisskopf.
\newblock New extended model of hadrons.
\newblock {\em Phys. Rev. D (3)}, 9(12):3471--3495, 1974.

\bibitem{PhysRevD.12.2060}
T.~DeGrand, R.~L. Jaffe, K.~Johnson, and J.~Kiskis.
\newblock Masses and other parameters of the light hadrons.
\newblock {\em Phys. Rev. D}, 12:2060--2076, Oct 1975.

\bibitem{evans1998partial}
L.~C. Evans.
\newblock {\em Partial differential equations}.
\newblock Providence, Rhode Land: American Mathematical Society, 1998.

\bibitem{johnson}
K.~Johnson.
\newblock The {MIT} bag model.
\newblock {\em Acta Phys. Pol.}, B(6):865--892, 1975.

\bibitem{stockmeyer2016infinite}
E.~Stockmeyer and S.~Vugalter.
\newblock {Infinite mass boundary conditions for Dirac operators}.
\newblock {\em arXiv preprint arXiv:1603.09657}, 2016.

\bibitem{Thaller1992}
B.~Thaller.
\newblock {\em The {D}irac equation}.
\newblock Texts and Monographs in Physics. Springer-Verlag, Berlin, 1992.

\end{thebibliography}

\end{document}